\documentclass[a4paper,12pt]{amsart}
\usepackage{amssymb}
\usepackage{hyperref}

\usepackage{tikz}
\usetikzlibrary{arrows,automata}

\tikzstyle{V}=[fill=black,circle,scale=0.4, outer sep = 4pt]

\newtheorem{thm}{Theorem}[section]
\newtheorem{prop}[thm]{Proposition}
\newtheorem{cor}[thm]{Corollary}
\newtheorem{lemma}[thm]{Lemma}

\theoremstyle{remark}
\newtheorem{rmk}[thm]{Remark}
\newtheorem{example}[thm]{Example}
\theoremstyle{definition}
\newtheorem{defn}[thm]{Definition}

\numberwithin{equation}{section}

\newcommand{\clsp}{\overline{\operatorname{span}}}

\newcommand{\llangle}{\mathopen{\langle\!\langle}}
\newcommand{\rrangle}{\mathclose{\rangle\!\rangle}}

\newcommand{\bi}{\begin{itemize}}
\newcommand{\ei}{\end{itemize}}
\newcommand{\be}{\begin{enumerate}}
\newcommand{\ee}{\end{enumerate}}

\newcommand{\C}{\mathbb{C}}

\newcommand{\T}{\mathbb{T}}

\newcommand{\K}{\mathcal{K}}
\newcommand{\Ll}{\mathcal{L}}
\newcommand{\M}{\mathcal{M}}

\newcommand{\N}{\mathbb{N}}
\newcommand{\Z}{\mathbb{Z}}

\newcommand{\hatimes}{\mathbin{\widehat{\otimes}}}

\newcommand{\psX}{X}

\title[Homotopy of product systems, and $K$-theory]{Homotopy of product systems and $K$-theory of Cuntz--Nica--Pimsner algebras}
\author{James Fletcher, Elizabeth Gillaspy, and Aidan Sims}

\address[J. Fletcher]{School of Mathematics and Statistics, Victoria University of Wellington, PO Box 600, Wellington 6140, New Zealand}
\email{jef336@uowmail.edu.au}

\address[E. Gillaspy]{Department of Mathematical Sciences, University of Montana, 32 Campus Drive \#0864, Missoula, MT 59812, USA}
\email{elizabeth.gillaspy@mso.umt.edu}

\address[A. Sims]{School of Mathematics and Applied Statistics, University of Wollongong, Northfields Ave Wollongong NSW 2522, Australia}
\email{asims@uow.edu.au}

\date{\today}

\subjclass{46L05}
\keywords{Product system; Cuntz--Nica--Pimsner; higher-rank graph; $K$-theory}

\thanks{We thank the anonymous referee for helpful comments and suggestions. This research was
supported in part by Marsden grant 15-UOO-071 from the Royal Society of New Zealand, by the Small
Grants Program of the University of Montana, by the US National Science Foundation
(grant DMS-1800749 to E.G.), and by Australian Research Council grant DP180100595.}

\begin{document}

\begin{abstract}
We introduce the notion of a homotopy of product systems, and show that the Cuntz--Nica--Pimsner
algebras of homotopic product systems over $\N^k$ have isomorphic $K$-theory. As an application,
we give a new proof that the $K$-theory of a $2$-graph $C^*$-algebra is independent of the
factorisation rules, and we further show that the $K$-theory of any twisted $2$-graph
$C^*$-algebra is independent of the twisting 2-cocycle. We also explore applications to
$K$-theory for the $C^*$-algebras of single-vertex $k$-graphs, reducing the question of whether
the $K$-theory is independent of the factorisation rules to a question about path-connectedness
of the space of solutions to an equation of Yang--Baxter type.
\end{abstract}

\maketitle

\section{Introduction}
The close link between the structure of a directed graph $E$ and that of its associated $C^*$-algebra $C^*(E)$, together with the structural restrictions on the $C^*$-algebras of directed graphs (for example \cite{KPR}, if $C^*(E)$ is simple, it must be either approximately finite-dimensional or purely infinite), spurred Kumjian and Pask to introduce higher-rank graphs ($k$-graphs) in \cite{kp2} as more general combinatorial models for $C^*$-algebras.
Since then, substantial work has
gone into the study of the structure of higher-rank graph $C^*$-algebras.
Significant progress has been made on properties like simplicity \cite{RS}, ideal structure
\cite{Si1, Si2}, pure infiniteness \cite{Si2}, stable finiteness and quasidiagonality \cite{CaHS},
and topological dimension and real rank \cite{PSS2}. However, there has been little progress  on
calculating the $K$-theory of a $k$-graph $C^*$-algebra since the initial work of
Robertson and Steger \cite{RobSt} on higher-rank Cuntz--Krieger algebras, and Evans'
generalisation \cite{Evans} of their work to the $C^*$-algebras of 2-graphs (higher-rank graphs of rank 2).

Higher-rank graphs, or $k$-graphs, are a $k$-dimensional generalization of directed graphs.  One can view a $k$-graph as a directed graph with $k$ colours of edges, together with a {\em factorisation rule} which gives  an equivalence relation  on the set of paths in the graph: each path from a vertex $v$ to another vertex $w$ consisting of a red edge followed by a blue edge (a {\em red-blue path}) must be equivalent to precisely one blue-red path from $v$ to $w.$ While the $C^*$-algebra of a $k$-graph depends on both the underlying  edge-coloured directed graph (its {\em skeleton}) and on the factorisation rule,
Evans' theorem shows that the $K$-theory of the $C^*$-algebra of a row-finite $2$-graph with no
sources depends only on the skeleton.

While  $k$-graph $C^*$-algebras can be used to realise many classes of $C^*$-algebras up to Morita
equivalence \cite{PRRS, ruiz-sims-sorensen, BOS}), many fundamental examples of $C^*$-algebras,
such as the rotation algebras $A_\theta$, cannot be realised as $k$-graph $C^*$-algebras
\cite[Corollary 5.7]{EvansSims}.  However, \cite[Example 7.7]{kps3} established that $A_\theta$ is
a {\em twisted} $k$-graph $C^*$-algebra. Furthermore, \cite{KPS5} and \cite{gillaspy-PJM}
identified sufficient conditions under which the $K$-theory of a twisted $k$-graph $C^*$-algebra
$C^*(\Lambda, c)$ agrees with that of its untwisted counterpart. Thus, substantial evidence
suggests that (at least when $k=2$) the $K$-theory of a (twisted) $k$-graph $C^*$-algebra should
depend only on the $k$-graph's skeleton.

Evans' techniques from \cite{Evans} also suggest the more general possibility that the $K$-theory of any $k$-graph
$C^*$-algebra is independent of the  factorisation rules,
and this was cast as a conjecture in \cite{BOS}. Intriguingly, this question arose in \cite{BOS}
not through the study of $k$-graphs themselves, but because single-vertex $k$-graphs arise
naturally as a framework to describe the ``torsion subalgebra" $\mathcal{A_S}$ of the
$C^*$-algebra $\mathcal{Q}_S$ associated to the multiplicative action on $\Z$ of the subsemigroup
of $\N^\times$ generated by a finite set $S$ of primes. In the single vertex case, every skeleton
admits one set of factorisation rules under which the associated $C^*$-algebra is isomorphic to a tensor product
$\bigotimes_{p \in S} \mathcal{O}_p$ of Cuntz algebras, but it also admits many other factorisation rules. Since the $C^*$-algebras
$\mathcal{Q}_S$ and $\mathcal{A}_S$ are UCT Kirchberg algebras \cite{BOS}, the arguments of
\cite{BOS} reduce the question of understanding the $C^*$-algebras $\mathcal{Q}_S$ to the question
of understanding $K_*(\mathcal{A}_S)$. It is therefore very interesting, even in the 1-vertex
case, to develop new techniques for investigating the conjecture that the $K$-theory of a
$k$-graph $C^*$-algebra does not depend on its factorisation rules.

In this paper, we present an approach to this problem that reduces it to a question about
path-connectedness of the space of unitary solutions to an equation of Yang--Baxter type (see, for example,
\cite{jimbo}). Our idea is based on Elliott's technique \cite{elliott} for computing the $K$-theory
of a noncommutative torus. Elliott's technique exploits the fact that the noncommutative tori of a
fixed rank assemble as the fibres of the $C^*$-algebra of a higher-rank integer Heisenberg group.
The base space for the fibration is a torus, hence path connected, so any two fibres can be connected
by a path, putting them at either end of a $C([0,1])$-algebra. This $C([0,1])$-algebra can be
viewed as a crossed product of $C([0,1])$ by $\Z^k$, and hence as an iterated crossed-product
$C([0,1]) \rtimes \Z \rtimes \cdots \rtimes \Z$. Since evaluation at each point in $[0,1]$ induces
an isomorphism $K_*(C([0,1])) \to K_*(\C)$, it therefore suffices, by induction, to prove that if
$A$ is a $C([0,1])$-algebra such that the quotient maps $A\mapsto A_t$ induce isomorphisms in
$K$-theory, and if $\alpha$ is an automorphism of $A$ that respects the $C([0,1])$-algebra
structure, then $A \rtimes \Z$ is again a $C([0,1])$-algebra in which the evaluation maps induce
isomorphisms in $K$-theory. Elliott proved this by applying the Five Lemma to the
Pimsner--Voiculescu exact sequence for the crossed product. A similar technique was used in
\cite{KPS5, gillaspy-PJM} to examine the $K$-theory of twisted $k$-graph $C^*$-algebras, and show that homotopic
cocycles yield twisted $k$-graph $C^*$-algebras with isomorphic $K$-theory.

Here, we employ a similar technique, but expand the notion of homotopy to which our results apply
by regarding product systems of Hilbert modules as generalised cocycles. Product systems (see
Section~\ref{sec:product-systems} below) and their $C^*$-algebras were introduced by Fowler in
\cite{Fowler}.  These families $\{ \psX_p\}_{p \in P}$ of Hilbert $A-A$-bimodules, indexed by a
semigroup $P$ and carrying a multiplication compatible with  that of $P$, give rise to a class of
$C^*$-algebras including $k$-graph $C^*$-algebras and crossed products by endomorphisms.

We introduce homotopies of product systems with coefficient algebra $A$, defined as product
systems $X$ with coefficient algebra $C([0,1], A)$ in which the canonical left and right actions
of $C([0,1])$ on each fibre coincide. Proposition~\ref{prp:homotopy equiv rel} verifies that this
notion of homotopy defines an equivalence relation on (isomorphism classes of) product systems,
while Lemma~\ref{lem:bundle of product systems} establishes that a homotopy $X$ of  product
systems decomposes into fibres $X^t$ indexed by $t \in [0,1]$ such that each $X^t$ is itself a
product system over $A$. Lemma~\ref{lem:bundle of product systems} also establishes  that if the
left action on a homotopy of product systems is injective and/or by compacts then the same is true
in each fibre. Using this, we verify in Proposition~\ref{pr:evaluation-maps-commute} that for
product systems over a quasi-lattice ordered semigroup in which each pair of elements has an upper
bound, the Cuntz--Nica--Pimsner algebra of a homotopy of product systems in which the left action
is injective and by compacts itself decomposes as a $C([0,1])$-algebra whose fibres are the
Cuntz--Nica--Pimsner algebras of the fibres of the homotopy. We then focus on product systems over
$\N^k$, and use Elliott's strategy described above, replacing the Pimsner--Voiculescu exact
sequence with Pimsner's six-term exact sequence for Cuntz--Pimsner algebras \cite{Pimsner}, to
show in Theorem~\ref{thm:htopy-KK} that if $X$ is a homotopy of product systems over $\N^k$, then
the quotient maps $\mathcal{NO}_X \to \mathcal{NO}_{X^t}$ all induce isomorphisms in $K$-theory.

We conclude the paper by applying this machinery to the setting of $k$-graphs. Given a $k$-graph
$\Lambda$, its $C^*$-algebra $C^*(\Lambda)$ is the Cuntz--Nica--Pimsner algebra of a product
system $X(\Lambda)$ over $\N^k$ (\cite[Proposition 5.4]{sims--yeend}; see also \cite{fs, RaeS}).
In fact, if $E_i$ is the graph with vertices $\Lambda^0$ and edges the edges of colour $i$ in
$\Lambda$, then the coordinate fibre of $X(\Lambda)$ over the $i$th generator $e_i$ of $\N^k$ is
the graph module $X(E_i)$. Thus, given two $k$-graphs with the same skeleton but different
factorisation rules, or the same skeleton and factorisation rules but different 2-cocycles, one naturally asks
whether  there is a homotopy of product systems linking them. In Section \ref{sec:k-graphs}, we
reduce this question to a question about the path-connectedness of a certain space of systems of
unitary matrices satisfying a cocycle-like condition. To be precise, we first show (using results
of \cite{fs}) that product systems over $\N^k$ with coordinate fibres isomorphic to the graph
modules $X(E_i)$ are determined by systems of unitary matrices $U_{i,j}(v,w) : \C^{v E^1_jE^1_i w}
\to \C^{v E^1_iE^1_j w}$, indexed by pairs $v,w$ of vertices and pairs $1 \le i < j \le k$ of
indices, that collectively satisfy a cocycle condition reminiscent of the Yang--Baxter equation.
We call such a system a {\em unitary cocycle} for $E$. For example, Proposition~\ref{prp:k-graph
to cocycle} establishes that for every row-finite $k$-graph $\Lambda$ with no sources, and every
$\T$-valued 2-cocycle $c$ on $\Lambda$, there is a unitary cocycle so that the
Cuntz--Nica--Pimsner algebra of the associated product system coincides with the twisted $k$-graph
$C^*$-algebra $C^*(\Lambda, c)$. We then show that a family $\{U^t: t \in [0,1]\}$ of unitary
cocycles for $E$ determines a homotopy of product systems if, for each fixed $u,v,i,j$, the map $t
\mapsto U^t_{i,j}(u,v)$ is continuous; see Proposition \ref{prp:path of cocycles}. We call such a
family a {\em continuous path} of unitary cocycles for $E$. Using our main result (Theorem
\ref{thm:htopy-KK}), we then deduce in Corollary~\ref{cor:connected->K-theory} that if
$\Lambda,\Gamma$ are $k$-graphs with the same skeleton $E$, and $c, c'$ are $\T$-valued 2-cocycles
on $\Lambda$ and $\Gamma$, and if the unitary cocycles for $E$ determined by $(\Lambda, c)$ and
$(\Gamma, c')$ are connected by a continuous path of unitary cocycles for $E$, then $C^*(\Lambda,
c)$ and $C^*(\Gamma, c')$ have isomorphic $K$-theory.

We deduce two main consequences. First, if $E = (E_1, E_2)$ is a skeleton of a 2-graph, then the cocycle
condition for a unitary cocycle for $E$ is vacuous, and so a unitary cocycle is simply a system of
unitary isomorphisms $U_{1,2}(v,w)$, indexed by $v,w \in E^0$, of finite-dimensional vector
spaces. Since the space of $n \times n$ complex unitary matrices is path-connected, we deduce that
the space of unitary cocycles for $E$ is path connected, and therefore that any two product
systems over $\N^2$ with coordinate fibres $X(E_1)$ and $X(E_2)$ are connected by a homotopy of
product systems. Hence the $K$-theory of the Cuntz--Nica--Pimsner algebra of any such product
system is determined by Evans' calculation \cite{Evans}
of $K$-theory for $2$-graph $C^*$-algebras. In
particular, if $\Lambda$ and $\Gamma$ are row-finite $2$-graphs with no sources and the same
skeleton, then $K_*(C^*(\Lambda, c)) \cong K_*(C^*(\Gamma, c'))$ for any $\T$-valued cocycles $c$
on $\Lambda$ and $c'$ on $\Lambda'$. Second, given positive integers $n_1, \dots, n_k$, a unitary
cocycle for the $k$-skeleton $(E_1, \dots E_k)$ in which each graph $E_i$ has one vertex and $n_i$ edges
reduces to a family of unitary matrices $U_{i,j} : \C^{n_j} \otimes \C^{n_i} \to \C^{n_i} \otimes
\C^{n_j}$, $i < j$, satisfying the Yang--Baxter type equations
\[
(U_{i,j} \otimes 1_l)(1_j \otimes U_{i,l})(U_{j,l} \otimes 1)
    = (1_i \otimes U_{j,l})(U_{i,l} \otimes 1_j)(1_l \otimes U_{i,j})
\]
for $1 \le i < j < l \le k$. In particular, if the collection of all such families of unitary
matrices is connected, then any twisted $C^*$-algebra of any $1$-vertex $k$-graph $\Lambda$ in
which each $|\Lambda^{e_i}| = n_i$ has $K$-theory isomorphic to $K_*(\bigotimes^k_{i=1}
\mathcal{O}_{n_i})$.

\smallskip

The structure of the paper is as follows. We summarize, in Section \ref{sec:background}, the
relevant background and known results on Hilbert bimodules and  product systems. In
Section~\ref{sec:main} we introduce the notion of homotopy of product systems, and prove our main
result, Theorem~\ref{thm:htopy-KK}, which states that the Cuntz--Nica--Pimsner algebras of
homotopic product systems have isomorphic $K$-theory. In Section~\ref{sec:k-graphs} we present
background on $k$-graphs and their twisted $C^*$-algebras, introduce the notion of a unitary
cocycle for a $k$-skeleton $E$, and use this to apply our earlier results to twisted $k$-graph
$C^*$-algebras.

Throughout the paper, we use the word ``homomorphism'' to denote a $*$-preserving, multiplicative, norm-decreasing linear map between $C^*$-algebras.

\section{Background on Hilbert bimodules and associated \texorpdfstring{$C^*$}{C*}-algebras}\label{sec:background}

We give a quick summary of the structure of Hilbert modules and their $C^*$-algebras. For details
on Hilbert modules, see \cite{Lance, tfb} and for details on the associated $C^*$-algebras, see
\cite{Pimsner, FR, Katsura}.

Let $A$ be a $C^*$-algebra. An {\em inner product $A$-module} is a complex vector space $X$ equipped
with a map $\langle \cdot, \cdot \rangle_A:X\times X\rightarrow A$, linear in its second argument,
and a right action of $A$, such that for any $x,y\in X$ and $a\in A$, we have
\begin{enumerate}
\item[(i)] $\langle x,y\rangle_A=\langle y,x\rangle_A^*$;
\item[(ii)] $\langle x,y\cdot a \rangle_A=\langle x,y\rangle_Aa$;
\item[(iii)] $\langle x,x\rangle_A\geq 0$ in $A$; and
\item[(iv)] $\langle x,x\rangle_A=0$ if and only if $x=0$.
\end{enumerate}
By \cite[Proposition~1.1]{Lance}, the formula $\|x\|_X:=\|\langle x,x\rangle_A\|_A^{1/2}$ defines
a norm on $X$, and we say that $X$ is a Hilbert $A$-module if $X$ is complete with respect to this
norm.

We say that a map $T:X\rightarrow X$ is {\em adjointable} if there exists a map $T^*:X\rightarrow
X$ such that $\langle Tx, y\rangle_A=\langle x, T^*y\rangle_A$ for all $x, y\in X$. Every
adjointable operator $T$ is automatically linear and continuous, and the adjoint $T^*$ is unique.
Equipping the collection of adjointable operators on $X$, denoted by $\mathcal{L}(X)$, with the
operator norm gives a $C^*$-algebra. For each $x,y\in X$, the formula $\Theta_{x,y}(z) := x\cdot
\langle y, z\rangle_A$ defines an adjointable operator with adjoint $\Theta_{x,y}^* =
\Theta_{y,x}$. The closed subspace $\mathcal{K}(X):=\overline{\mathrm{span}}\{\Theta_{x,y}:x,y\in
X\}$ is an essential ideal of $\mathcal{L}(X)$, whose elements we refer to as compact operators.

A {\em Hilbert $A$-bimodule} is a Hilbert $A$-module $X$ equipped with a left action of $A$ by
adjointable operators (i.e.~a homomorphism $\phi:A\rightarrow \mathcal{L}(X)$). We frequently
write $a\cdot x$ for $\phi(a)(x)$. Letting $A$ act on itself by left and right multiplication, and
defining an $A$-valued inner product on $A$ by $\langle a,b\rangle_A:=a^*b$, gives a Hilbert
$A$-bimodule that we denote by ${}_A A_A$.

By the Hewitt--Cohen factorisation theorem, every Hilbert $A$-bimodule $X$ is automatically
right-nondegenerate in the strong sense that $X = X \cdot A$. So we will say that the Hilbert
$A$-bimodule is \emph{nondegenerate} if the homomorphism $\phi : A \to \Ll(X)$ that implements the
left action is nondegenerate; that is if $X = \overline{\phi(A)X}$.

\begin{example}\label{ex:graph module}
Let $E = (E^0, E^1, r, s)$ be a row-finite graph with no sources. Let $A := C_0(E^0)$. Define
$\langle \cdot, \cdot \rangle_A$ on $C_c(E^1)$ by $\langle \xi, \eta\rangle_A(v) = \sum_{e \in E^1
v} \overline{\xi(e)}\eta(e)$, and define left and right actions of $A$ on $C_c(E^1)$ by $(a \cdot
\xi \cdot b)(e) = a(r(e))\xi(e)b(s(e))$. Then $\|\xi\| := \|\langle \xi, \xi\rangle_A\|^{1/2}$
defines a norm on $C_c(E^1)$. The completion $X = X(E)$ of $C_c(E)$ in this norm is a Hilbert
$A$-bimodule with $A$-actions extending those on $C_c(E^1)$. This is called the \emph{graph
bimodule} of $E$.
\end{example}

A {\em Toeplitz representation} $(\psi,\pi)$ of a Hilbert $A$-bimodule $X$ in a $C^*$-algebra $B$
consists of a linear map $\psi:X\rightarrow B$ and a homomorphism $\pi:A\rightarrow B$ such that
\begin{enumerate}
\item[(i)] $\psi(a\cdot x)=\pi(a)\psi(x)$ for each $a\in A$, $x\in X$;
\item[(ii)] $\psi(x)^*\psi(y)=\pi(\langle x,y\rangle_A)$ for each $x,y\in X$.
\end{enumerate}
These relations imply that $\psi(x \cdot a) = \psi(x)\pi(a)$ for all $a \in A$ and $x \in X$, and
also that $\psi$ is norm-decreasing, and is isometric if and only $\pi$ is injective on $\langle
X, X\rangle_A$. The universal $C^*$-algebra for Toeplitz representations of $X$ is called the
{\em Toeplitz algebra} of $X$. We write $\mathcal{T}_X$ for this $C^*$-algebra and denote the universal
Toeplitz representation of $X$ that generates it by $(i_X,i_A)$. By
\cite[Proposition~8.11]{Raeburn}, if $(\psi, \pi)$ is a Toeplitz representation of $X$ in $B$,
then there is a homomorphism $(\psi,\pi)^{(1)}:\mathcal{K}(X)\rightarrow B$ such that
$(\psi,\pi)^{(1)}\left(\Theta_{x,y}\right)=\psi(x)\psi(y)^*$ for all $x,y\in X$. We say that a
Toeplitz representation $(\psi,\pi)$ is {\em Cuntz--Pimsner covariant} if
$(\psi,\pi)^{(1)}(\phi(a))=\pi(a)$ for all $a\in J_X:=\phi^{-1}(\mathcal{K}(X))\cap
\mathrm{ker}(\phi)^\perp$.\footnote{Given an ideal $I$ of a $C^*$-algebra $A$, we write $I^\perp$
for the annihilator $\{a \in A : aI = 0\}$.} We call the universal $C^*$-algebra for
Cuntz--Pimsner covariant Toeplitz representations of $X$ the Cuntz--Pimsner algebra of $X$. We
denote this $C^*$-algebra by $\mathcal{O}_X$ and write $(j_X,j_A)$ for the universal
Cuntz--Pimsner covariant Toeplitz representation of $X$.

The $K$-theory of the Toeplitz algebra of a Hilbert $A$-bimodule $X$ is easy to compute: by
\cite[Proposition~8.2]{Katsura} the homomorphism $i_A:A\rightarrow \mathcal{T}_X$ induces an
isomorphism at the level of $K$-theory (in fact if $A$ is separable and $X$ is countably
generated, then this homomorphism induces a $KK$-equivalence \cite[Theorem~4.4]{Pimsner}). In
general, the $K$-theory of the Cuntz--Pimsner algebra $\mathcal{O}_X$ is much more complicated;
the primary tool for computing it is the following $6$-term exact sequence
\cite[Theorem~8.6]{Katsura}
\[
\begin{tikzpicture}
    \node (00) at (0,0) {$K_1(\mathcal{O}_X)$};
    \node (40) at (4,0) {$K_1(A)$};
    \node (80) at (8,0) {$K_1(J_X)$.};
    \node (82) at (8,2) {$K_0(\mathcal{O}_X)$};
    \node (42) at (4,2) {$K_0(A)$};
    \node (02) at (0,2) {$K_0(J_X)$};
    \draw[-stealth] (02)-- node[above] {$\iota_*-[X]$} (42);
    \draw[-stealth] (42)-- node[above] {$(j_A)_*$} (82);
    \draw[-stealth] (82)-- node[right] {$$} (80);
    \draw[-stealth] (80)-- node[below] {$\iota_*-[X]$} (40);
    \draw[-stealth] (40)-- node[below] {$(j_A)_*$} (00);
    \draw[-stealth] (00)-- node[left] {$$} (02);
\end{tikzpicture}
\]

Given Hilbert $A$-bimodules $X$ and $Y$, we can form the balanced tensor product $X\otimes_A Y$
as follows (see \cite{Lance} or \cite{tfb}). We endow the algebraic tensor product $X \odot Y$
with the canonical actions of $A$ given by $a \cdot (x \odot y) \cdot b = (a \cdot x) \odot (y
\cdot b)$, and with the sesquilinear form $[x \odot y, x' \odot y']_A := \langle y, \langle x,
x'\rangle_A \cdot y'\rangle_A$. The space $N = \{\zeta \in X \odot Y : [\zeta, \zeta]_A = 0\}$ is
a closed submodule, and $[\cdot, \cdot]_A$ descends to an inner-product $\langle \cdot,
\cdot\rangle_A$ on $(X \odot Y)/N$. The balanced tensor product $X \otimes_A Y$ is the completion
of $(X \odot Y)/N$ in the inner-product norm, which is itself a Hilbert $A$-bimodule. We write $x
\otimes y$ for the image $x \odot y + N$ of $x \odot y$ in $X \otimes_A Y$. A simple calculation
shows that $[x\cdot a \odot y, \zeta]_A = [x \odot a\cdot y, \zeta]$ for all $x \in X$, $y \in Y$,
$a \in A$ and $\zeta \in X \odot Y$, and it follows that $x \cdot a \otimes y = x \otimes a \cdot
y$ for all $x,y,a$. If $\phi : A \to \Ll(Y)$ is the homomorphism implementing the left action, we
sometimes write $X \otimes_\phi Y$ rather than $X \otimes_A Y$.

We define the balanced tensor powers of $X$ as follows: $X^{\otimes 0}:={}_A A_A$, $X^{\otimes
1}:=X$, and $X^{\otimes n}:=X\otimes_A X^{\otimes n-1}$ for $n\geq 2$. Given Hilbert $A$-bimodules
$X$ and $Y$ and an adjointable operator $S\in \mathcal{L}(X)$, the formula $x\otimes y \mapsto
(Sx)\otimes y$ extends to an adjointable operator on all of $X\otimes_A Y$, which we denote by
$S\otimes_A 1_Y$, or by $S \otimes_\phi 1_Y$ if $\phi : A \to \Ll(X)$ is the homomorphism
implementing the left action.

We can also combine a collection $\{X_j:j\in J\}$ of Hilbert $A$-bimodules by forming their direct
sum. We define $\bigoplus_{j\in J}X_j$ to be the set of sequences $(x_j)_{j\in J}$ with $x_j\in
X_j$, such that $\sum_{j\in J}\langle x_j, x_j\rangle_A$ converges in $A$. We define an $A$-valued
inner product on $\bigoplus_{j\in J}X_j$ by $\left\langle (x_j)_{j\in J}, (y_j)_{j\in
J}\right\rangle_A:=\sum_{j\in J}\langle x_j, y_j\rangle_A$, which converges by
\cite[Proposition~1.1]{Lance}. It follows that $\bigoplus_{j\in J}X_j$ is complete with respect to
the norm induced by this inner product. Letting $A$ act componentwise on the left and right then
gives $\bigoplus_{j\in J}X_j$ the structure of a Hilbert $A$-bimodule.

\subsection{Product systems of Hilbert bimodules and their associated \texorpdfstring{$C^*$}{C*}-algebras}
\label{sec:product-systems}
A {\em quasi-lattice ordered group} $(G,P)$ consists of a group $G$ and a subsemigroup $P$ of $G$ such
that $P\cap P^{-1}=\{e\}$, and such that, with respect to the partial order on $G$ given by $p\leq
q \Leftrightarrow p^{-1}q\in P$, any two elements $p,q\in G$ which have a common upper bound in
$P$ have a least common upper bound $p \vee q$ in $P$. We write $p\vee q=\infty$ if $p$ and $q$
have no common upper bound in $P$, and $p\vee q<\infty$ otherwise. We say that $P$ is directed if
$p\vee q<\infty$ for every $p,q\in P$.

Let $(G,P)$ be a quasi-lattice ordered group and $A$ a $C^*$-algebra. A {\em product system} over
$P$ with coefficient algebra $A$ is a semigroup $\psX=\bigsqcup_{p\in P}\psX_p$ such that:
\begin{enumerate}
\item[(i)] $\psX_e = {_A A_A}$, and $\psX_p\subseteq \psX$ is a Hilbert $A$-bimodule for each
    $p\in P$;
\item[(ii)] for each $p, q \in P$ with $p \not= e$, there exists a Hilbert $A$-bimodule
    isomorphism $M_{p,q}:\psX_p\otimes_A \psX_q\rightarrow \psX_{pq}$ which is associative in
    the sense that
    \[ M_{r, pq} \circ 1_{\psX_r} \otimes M_{p,q} = M_{rp, q} \circ M_{r,p} \otimes 1_{\psX_q}\]
     for each $p,q, r\in P$;
     and
\item[(iii)] multiplication in $\psX$ by elements of $\psX_e=A$ implements the left and right
    actions of $A$ on each $\psX_p$; that is $xa=x\cdot a$ and $ax=a\cdot x$ for each $p\in P$,
    $a\in A$, and $x\in \psX_p$.
\end{enumerate}
We will often write $M_{p,q}(x \otimes y) =: xy$.

We will say that the product system $\psX$ is \emph{nondegenerate} if each fibre $\psX_p$ is
nondegenerate as a Hilbert $A$-bimodule; that is, if $\psX_p = \overline{\phi_p(A)\psX_p}$ for all
$p$. Note that since the multiplication maps $\psX_e \times \psX_p \to \psX_p$ are given by the
left action, the product system $\psX$ is nondegenerate if and only if condition~(ii) above also
holds when $p = e$.

We write $\phi_p:A\rightarrow \mathcal{L}(\psX_p)$ for the homomorphism that implements the left
action of $A$ on $\psX_p$. Since multiplication in $\psX$ is associative,
$\phi_{pq}(a)(xy)=(\phi_p(a)x)y$ for all $p,q\in P$, $a\in A$, $x\in \psX_p$, and $y\in \psX_q$.
We write $\langle \cdot,\cdot\rangle_A^p$ for the $A$-valued inner-product on $\psX_p$.

For each $p,q\in P$, we define a homomorphism
$\iota_p^{pq}:\mathcal{L}\left(\psX_p\right)\rightarrow \mathcal{L}\left(\psX_{pq}\right)$ by
\[
\iota_p^{pq}(S):=M_{p,q}\circ (S\otimes_A 1_{\psX_q})\circ M_{p,q}^{-1}
\]
for each $S\in \mathcal{L}\left(\psX_p\right)$. Equivalently, $\iota_p^{pq}$ is characterised by
the formula $\iota_p^{pq}(S)(xy)=(Sx)y$ for each $S\in \mathcal{L}\left(\psX_p\right)$, $x\in
\psX_p$, $y\in \psX_q$. We define $\iota_p^r: \mathcal{L}\left(\psX_p\right)\rightarrow
\mathcal{L}\left(\psX_r\right)$ to be the zero map whenever $p\not\leq r$. We say that $\psX$ is
{\em compactly aligned} if $\iota_p^{p\vee q}(S)\iota_q^{p\vee q}(T)\in \mathcal{K}(\psX_{p\vee q})$
whenever $S\in \mathcal{K}(\psX_p)$ and $T\in \mathcal{K}(\psX_q)$ for some $p,q\in P$ with $p\vee
q<\infty$. By \cite[Proposition~5.8]{Fowler}, if $A$ acts compactly on each fibre of $\psX$ (i.e.
$\phi_p(A)\subseteq \mathcal{K}(\psX_p)$ for each $p\in P$) or $(G,P)$ is totally ordered by
$\leq$, then $\psX$ is automatically compactly aligned.

A representation of a compactly aligned product system $\psX$ over $P$ in a $C^*$-algebra $B$ is a
map $\psi:\psX\rightarrow B$ such that
\begin{enumerate}
\item[(i)] each $\psi_p:=\psi|_{\psX_p}$ is a linear map, and $\psi_e$ is a homomorphism;
\item[(ii)] $\psi_p(x)\psi_q(y)=\psi_{pq}(xy)$ for all $p,q\in P$ and $x\in \psX_p$, $y\in
    \psX_q$; and
\item[(iii)] $\psi_p(x)^*\psi_p(y)=\psi_e(\langle x,y\rangle_A^p)$ for all $p\in P$ and $x,y\in
    \psX_p$.
\end{enumerate}
For each $p\in P$, it follows that $(\psi_p,\psi_e)$ is a Toeplitz representation of the Hilbert
$A$-bimodule $\psX_p$. We write $\psi^{(p)}$ for the resulting homomorphism
$(\psi_p,\psi_e)^{(1)}:\mathcal{K}\left(\psX_p\right)\rightarrow B$.

If $\rho : B \to C$ is a homomorphism of $C^*$-algebras and $\psi : \psX \to B$ is a
representation, then $\rho \circ \psi : \psX \to C$ is a representation. For $x,y \in \psX_p$, we
have $(\rho \circ \psi)^{(p)}(\theta_{x,y}) = \rho(\psi(x)) \rho(\psi(y))^* =
\rho(\psi(x)\psi(y)^*) = \rho \circ \psi^{(p)}(\theta_{x,y})$. So linearity and continuity give
\begin{equation}\label{eq:induced composition}
(\rho \circ \psi)^{(p)} = \rho \circ \psi^{(p)}.
\end{equation}

We say that $\psi$ is {\em Nica covariant} if, for any $p,q\in P$ and $S\in \mathcal{K}(\psX_p)$, $T\in
\mathcal{K}(\psX_q)$,
\begin{align*}
\psi^{(p)}(S)\psi^{(q)}(T)=
\begin{cases}
\psi^{(p\vee q)}\left(\iota_p^{p\vee q}(S)\iota_q^{p\vee q}(T)\right) & \text{if $p\vee q<\infty$}\\
0 & \text{otherwise.}
\end{cases}
\end{align*}
We denote the universal $C^*$-algebra for Nica covariant representations by $\mathcal{NT}_\psX$
(the {\em Nica--Toeplitz algebra} of $\psX$) and write $ \{ i_p :\psX_p \rightarrow \mathcal{NT}_\psX\}_{p \in P}$ for
its generating representation. It then follows from relations (i)--(iii) that
$\mathcal{NT}_\psX=\overline{\mathrm{span}}\left\{i_\psX(x)i_\psX(y)^*:x,y\in \psX\right\}$.

For representations of product systems we also have a notion of Cuntz--Pimsner covariance (first
introduced by Sims and Yeend in \cite{sims--yeend}).
To formulate this covariance relation we
first require some additional definitions.
We start by defining a collection of ideals of $A$ by
setting $I_e:=A$ and $I_p:=\bigcap_{e<q\leq p}\ker(\phi_q)$ for each $p\in P\setminus \{e\}$. For
each $p\in P$, we then define a Hilbert $A$-bimodule
\[
\widetilde{\psX}_p:=\bigoplus_{q\leq p} \psX_q\cdot I_{q^{-1}p},
\]
and write $\widetilde{\phi}_p:A\rightarrow \mathcal{L}\big(\widetilde{\psX}_p\big)$ for the
homomorphism defined by
\[
\big(\widetilde{\phi}_p(a)(x)\big)_q:=\phi_q(a)(x_q)
\quad
\text{for $a\in A$, $x\in \widetilde{\psX}_p$, $q\leq p$.}
\]
We say that $\psX$ is $\widetilde{\phi}$-injective if each homomorphism $\widetilde{\phi}_p$ is
injective. For various examples (in particular, for product systems over $\N^k$, or where each
$\phi_p$ is injective) $\widetilde{\phi}$-injectivity is automatic by
\cite[Lemma~3.15]{sims--yeend}.
For each $p,q\in P$, we  have a homomorphism
$\widetilde{\iota}_p^{\,q}:\mathcal{L}(\psX_p)\rightarrow \mathcal{L}\big(\widetilde{\psX}_q\big)$
characterised by the formula
\[
\big(\widetilde{\iota}_p^{\,q}(S)(x)\big)_r=\iota_p^r(S)(x_r)
\quad
\text{for $S\in \mathcal{L}(\psX_p)$, $x\in \widetilde{\psX}_p$, $r\leq q$.}
\]
 Given a predicate statement $\mathcal{P}(s)$ (where
$s\in P$), we say that $\mathcal{P}(s)$ is true for large $s$ if, given any $p\in P$, there exists
$q\geq p$, such that $\mathcal{P}(s)$ is true whenever $s\geq q$. We then say that a Nica covariant representation  $\psi$ of a compactly aligned $\widetilde \phi$-injective product system $\psX$ is
{\em Cuntz--Pimsner covariant} if, for any finite set $F\subseteq P$ and any choice of compact operators
$\left\{T_p\in \mathcal{K}\left(\psX_p\right):p\in F\right\}$, we have that
\[
\sum_{p\in F} \widetilde{\iota}_p^{\,s} (T_p)=0\in \mathcal{L}\big(\widetilde{\psX}_s\big) \quad \text{for large $s$} \quad \Rightarrow \quad \sum_{p\in F} \psi^{(p)}(T_p)=0.
\]
We say that $\psi$ is Cuntz--Nica--Pimsner covariant if it is both Nica covariant and
Cuntz--Pimsner covariant.

If $P$ is directed and each $\phi_p$ is injective and takes values in $\K(X_p)$, then
\cite[Corollary~5.2]{sims--yeend} shows that a representation $\psi$ of $\psX$ is
Cuntz--Nica--Pimsner covariant if and only if $\psi^{(p)}\circ\phi_p=\psi_e$ for each $p\in P$.

We denote the universal $C^*$-algebra for  Cuntz--Nica--Pimsner covariant covariant
representations by $\mathcal{NO}_\psX$ (the {\em Cuntz--Nica--Pimsner algebra} of $\psX$) and
write $\{j_p:\psX_p \rightarrow \mathcal{NO}_\psX\}_{p \in P}$ for its generating representation.
It follows that $\mathcal{NO}_\psX$ is a quotient of $\mathcal{NT}_\psX$, and we write
$q:\mathcal{NT}_\psX\rightarrow \mathcal{NO}_\psX$ for the quotient homomorphism (characterised by
$q\circ i_\psX=j_\psX$). We have
$\mathcal{NO}_\psX=\overline{\mathrm{span}}\left\{j_\psX(x)j_\psX(y)^*:x,y\in \psX\right\}$.

\section{Homotopies of product systems}\label{sec:main}

To define our notion of a homotopy of product systems, we begin with a little background. Suppose
that $A$ is a $C^*$-algebra, and that $X$ is a right-Hilbert $(A \otimes C([0,1]))$-module.
Identifying $A \otimes C([0,1])$ with $C([0,1], A)$ as usual, Cohen factorisation shows that each
$x \in X$ can be written as $x = y \cdot f$ for some $y \in X$ and $f \in C([0,1], A)$. It follows
that there is an action of $C([0,1])$ on the right of $X$ such that $(x \cdot f) \cdot g = x \cdot
(fg)$ for all $x \in X$, $f \in C([0,1], A)$ and $g \in C([0,1])$.

With $X$ and $A$ as above, for each $t \in [0,1]$, we write $I_t$ for the ideal $A \otimes
C_0([0,1] \setminus\{t\}) \subseteq A \otimes C([0,1])$, and we write $X^t$ for the quotient
right-Hilbert $A$-module $X/(X \cdot I_t)$. It is standard that there is a unique topology on
$\mathcal{X} := \bigsqcup_t X^t$ under which $\mathcal{X}$ is a continuous Banach bundle in which,
for each $x \in X$, the map
\[
\gamma_x : t \mapsto x + X \cdot I_t
\]
is continuous. With respect to this topology, the map $X \owns x \mapsto \gamma_x \in
\Gamma([0,1], \mathcal{X})$ is an isomorphism of $X$ onto the module of continuous sections of
$\mathcal{X}$.\footnote{For example, one could prove this by applying
\cite[Theorem~C.26]{Williams} to the linking algebra $\big(\begin{smallmatrix} \K(X) & X \\ X^* &
A\end{smallmatrix}\big)$, and then take the sub-bundle of the resulting bundle of $C^*$-algebras
consisting of sections corresponding to elements of the form $\big(\begin{smallmatrix} 0 & \xi
\\ 0 & 0\end{smallmatrix}\big)$. This is a continuous, rather than upper-semicontinous, bundle
because $t \mapsto \|f(t)\|$ is continuous for $f \in C([0,1],A)$, and so for $x \in X$, the map
$t \mapsto \|x + X \cdot I_t\|^2 = \|\langle x, x\rangle_{A \otimes C([0,1])}(t)\|$ is
continuous.}

Now suppose that $X$ is a nondegenerate Hilbert bimodule over $A \otimes C([0,1])$. Then $\phi$
extends to a homomorphism from $\M(A \otimes C([0,1])) = \M(A) \otimes C([0,1])$ to $\Ll(X)$, and
in particular determines a left action of $C([0,1])$ on $X$. We say that $X$ is \emph{fibred over
$[0,1]$} if $f \cdot x = x \cdot f$ for all $x \in X$ and $f \in C([0,1])$. If $X$ is fibred over
$[0,1]$, then for each $t \in [0,1]$, the right-Hilbert module $X^t$ becomes a Hilbert
$A$-bimodule with left action satisfying $a \cdot (x + I_t) = (f \cdot x) + I_t$ for any $f \in
C([0,1], A)$ satisfying $f(t) = a$.

If $P$ is a semigroup, and $X$ is a product system over $P$ with coefficient algebra $A \otimes
C([0,1])$, we will say that $X$ is \emph{fibred over $[0,1]$} if each $X_p$ is fibred over
$[0,1]$.

\begin{lemma}\label{lem:bundle of product systems}
Let $P$ be a semigroup, and let $X$ be a nondegenerate product system over $P$ with coefficient
algebra $C([0,1], A)$ that is fibred over $[0,1]$. For each $t \in [0,1]$ the system $X^t :=
\{X_p^t : p \in P\}$ is a product system over $P$ with coefficient algebra $A$, with
multiplication given by $(x + X_p \cdot I_t)(y + X_q \cdot I_t) = xy + X_{pq} \cdot I_t$ for all
$p,q \in P$, $x \in X_p$ and $y \in X_q$. For any $p \in P$, if the left action of $C([0,1], A)$
on $X_p$ is by compacts, then the left action of $A$ on each $X^t_p$ is by compacts. For any $p
\in P$, if the action of $A$ on each $X^t_p$ is faithful then the action of $C([0,1], A)$ on $X_p$
is faithful; and if the left action of $C([0,1], A)$ on $X_p$ is by compacts, then the converse
holds.
\end{lemma}

Our proof of this lemma uses one direction of a result of \cite{KPS7} which is stated there as an
if and only if. The other implication is flawed, so we pause to explain the issue and why the
implication we need to use is nevertheless correct. This text originated in a sequel
to~\cite{KPS7} currently in preparation by Patterson, Sierakowski, Taylor and the third author; we
thank the other three authors for allowing us to publish it here instead.

\begin{rmk}\label{rmk:KPS7}
In our next proof, we use the ``$\Rightarrow$'' implication of \cite[Lemma~A.2]{KPS7}. The lemma
asserts that given a right-Hilbert $C([0,1], A)$-module $X$, an element $T \in \Ll(X)$ is compact
if and only if each $\widetilde{\varepsilon}_t(T)$ is compact. This is incorrect: the
``$\Leftarrow$'' implication fails even for full right-Hilbert $C([0,1], A)$-modules: take, for
example, $A = \C$ and let $X = C_0([0,1)) \oplus C_0((0,1])$ under its natural $C([0,1])$-valued
inner-product. Then $\widetilde{\varepsilon}_t(1_X)$ is compact for all $t$, but $1_X$ is not
compact. However, the ``$\Rightarrow$'' implication---that if $T \in \K(X)$, then each
$\widetilde{\varepsilon}_t(T) \in \K(X \otimes_{\varepsilon_t} A)$---and the proof of this
implication given in \cite{KPS7} are correct; and this is the implication that we invoke below.

Fortunately, the error  mentioned above has no flow-on effects in \cite{KPS7}. The incorrect characterisation
of compact operators on a right-Hilbert $C([0,1], A)$-module is only used in the proof of
\cite[Proposition~A.1]{KPS7} to show that if $X, Y$ are right-Hilbert $C([0,1], A)$-modules and $U
: X \otimes_{\varepsilon_1} A \to Y \otimes_{\varepsilon_0} A$ is an isomorphism, and if $S \in
\K(X)$ and $T \in \K(Y)$ satisfy $U \widetilde{\varepsilon}_1(S) U^{-1} =
\widetilde{\varepsilon}_0(T)$, then the operator $S \oplus T \in \K(X \oplus Y)$ restricts to a
compact operator on the right-Hilbert $C([0,1], A)$-module $Z = \{(\xi,\eta) \in X \oplus Y :
U(\xi \otimes_{\varepsilon_1} 1) = \eta \otimes_{\varepsilon_0} 1\}$ with $\xi\mapsto \xi
\otimes_{\varepsilon_t} 1$ denoting the quotient map of $X$ into $X \otimes_{\varepsilon_t} A\cong
X/\{x: \varepsilon_t(\langle x,x\rangle_A)=0\}$ and similarly for $\eta\mapsto \eta
\otimes_{\varepsilon_t} 1$. A correct proof of this statement appears as \cite[Lemma~4.5]{Kaad} (a
more direct proof that does not invoke Kasparov's stabilisation theorem is also possible), and
this fixes the gap in the proof of \cite[Proposition~A.1]{KPS7}.
\end{rmk}

\begin{proof}[Proof of Lemma~\ref{lem:bundle of product systems}]
Fix $t \in [0,1]$. Each $X_t$ is a Hilbert $A$-bimodule as discussed above, so we just have to
show that the formula $(x + X_p \cdot I_t)(y + X_q \cdot I_t) = xy + X_{pq} \cdot I_t$ determines
a well-defined multiplication that induces isomorphisms $X^t_p \otimes_A X^t_q \cong X^t_{pq}$. For
$p,q \in P$ let $M_{p,q} : X_p \otimes_{C([0,1], A)} X_q \to X_{pq}$ be the multiplication map.
Fix $x \in X_p \cdot I_t$, and write $x = x' \cdot f$ where $f(t) = 0$. Since $M_{p,q}$ is an
isomorphism and since $X_q$ is fibred over $[0,1]$, for any $y \in X^t_q$, we have
\begin{align*}
xy &= M_{p,q}(x \otimes y)
    = M_{p,q}(x' \otimes f \cdot y)\\
    &= M_{p,q}\big((x' \otimes y) \cdot f\big)
    = M_{p,q}(x' \otimes y)\cdot f = x'y \cdot f \in X_{pq} \cdot I_t.
\end{align*}
So if $x + X_p \cdot I_t = x' + X_p \cdot I_t$ then $xy - x'y = (x - x')y \in X_{pq} \cdot I_t$.
A simple argument using associativity and distributivity of multiplication shows that if $y + X_q
\cdot I_t = y' + X_q \cdot I_t$, then for any $x \in X_q$ we have $xy = xy'$, and it follows that
the formula for multiplication is well-defined.

For $x, x' \in X_p$ and $y, y' \in X_q$, blurring, where appropriate, the distinction betwen
$C([0,1], A)/I_t$ and $A$, we have
\begin{align*}
\langle xy + X_{pq} \cdot I_t, {}&x'y' + X_{pq} \cdot I_t\rangle_A\\
    &= \langle xy, x'y'\rangle_{C([0,1], A)} + I_t\\
    &= \big\langle x \otimes y, x' \otimes y'\big\rangle_{C([0,1], A)} + I_t\\
    &= \big\langle x \otimes y + (X_p \otimes_{C([0,1], A)} X_q) \cdot I_t,\\
    &\hskip8em x' \otimes y' + (X_p \otimes_{C([0,1], A)} X_q) \cdot I_t\big\rangle_{C([0,1], A)} + I_t\\
    &= \big\langle (x + X_p \cdot I_t) \otimes (y + X_q \cdot I_t), (x' + X_p \cdot I_t) \otimes (y' + X_q \cdot I_t)\big\rangle_A.
\end{align*}
So the multiplication map described above defines isomorphisms $X^t_p \otimes_A X^t_q \cong
X^t_{pq}$ of Hilbert modules. Associativity follows from associativity of multiplication in $X$.

For the second-last assertion, first recall from the ``$\Rightarrow$'' implication of
\cite[Lemma~A.2]{KPS7} (see Remark~\ref{rmk:KPS7}) that if $T \in \K (X_p)$ is compact and
$\varepsilon_t: C([0,1], A) \to A$ is induced by evaluation at $t$, then $T
\otimes_{\varepsilon_t} 1$ belongs to $\mathcal{K}(X_p \otimes_{\varepsilon_t} A)$ for all $t \in
[0,1]$. The map $x \otimes a \mapsto (x + X \cdot I_t) \cdot a$ is an isomorphism $X_p
\otimes_{\varepsilon_t} A \cong X^t_p$ that intertwines $\phi_p(f) \otimes 1$ with
$\phi^t_p(f(t))$ for $f \in C([0,1], A)$ and $t \in [0,1]$. So we deduce that if $\phi_p(f)$ is
compact then each $\phi^t_p(f(t))$ is compact. Thus if $A \otimes C([0,1])$ acts by compacts on
$X_p$, then $A$ acts by compacts on each $X^t_p$.

For the last assertion, first note that if each $\phi^t_p$ is injective, then each $\phi^t_p$ is
isometric, and so $\|\phi_p(f)\| = \sup_{t \in [0,1]} \|\phi_p^t(f(t))\| = \sup_{t \in [0,1]}
\|f(t)\| = \|f\|$. For the converse implication, suppose that $C([0,1], A)$ acts by compacts on
$X_p$. We will first show that for any $f \in C([0,1], A)$ and any $p \in P$, the map $t \mapsto
\| \phi_p^t(f(t))\|$ is continuous. To this end, note that the second paragraph of the proof of
\cite[Lemma~A.2]{KPS7} shows that, writing $I_{X_p} := \langle X_p, X_p\rangle_{C([0,1], A)}
\subseteq C([0,1], A)$, the module $X_p$ is a $\mathcal{K}(X_p)$--$I_{X_p}$-imprimitivity
bimodule, and the open map $\operatorname{Prim}(\mathcal{K}(X_p)) \to [0,1]$ arising from the
$C([0,1])$-structure as in \cite[Theorem~C.26]{Williams} is the composition of the Rieffel
homeomorphism for $X_p$ with the canonical map $\operatorname{Prim}(C([0,1], A)) \to [0,1]$. Since
this final map  is open, so is the composition, and we deduce from the final assertion of
\cite[Theorem~C.26]{Williams} that $\mathcal{K}(X_p)$ is a $C([0,1])$-algebra with norm-continuous
sections. Consequently, if $T \in \K(X_p)$, the map $t \mapsto \| T \otimes_{\varepsilon_t} 1\|$
is continuous.  Since we are assuming that the left action of $C([0,1], A)$ on $X_p$ is by
compacts, we see that $t \mapsto \|\phi^t_p(f(t))\|$ is continuous for each $f \in C([0,1], A)$.

Now suppose that there exists $t \in [0,1]$ such that $\phi^t_p$ is not injective. Then there
exists $a \in A$ with $\|a\| = 1$ such that $\phi^t_p(a) = 0$. By continuity of the norm, there
exists an open neighbourhood $U$ of $t$ in $[0,1]$ such that $\|\phi_p^s(a)\| < 1/2$ for all $s
\in U$. Choose $f \in C_0(U) \subseteq C([0,1])$ with $0 \le f \le 1$ and $f(t) = 1$. The function
$fa \in C([0,1], A)$ given by $(fa)(s) = f(s)a$ satisfies $\|fa\| = \|a\| = 1$, and
$\|\phi_p(fa)\| = \sup_{s \in [0,1]} \|\phi^s_p((fa)(s))\| = \sup_{s \in U} f(s)\|\phi_p^s(a)\| <
1/2$. So $\phi_p$ is not isometric and therefore not injective.
\end{proof}

The previous Lemma enables us to adopt the following definition of a homotopy of product systems.

\begin{defn} \label{def:htopy-prod-syst}
Let $A$ be a $C^*$-algebra, and let $P$ be a semigroup. Let $Y$ and $Z$ be product systems over
$P$ with coefficient algebra $A$. A \emph{homotopy of product systems} from $Y$ to $Z$ is a
nondegenerate product system $X$ over $P$ with coefficient algebra $C([0,1], A)$ that is fibred
over $[0,1]$ such that the product system $X^0$ is isomorphic to $Y$ and the product system $X^1$
is isomorphic to $Z$. We say that $Y$ and $Z$ are \emph{homotopic}.
\end{defn}

\begin{example}\label{eg:trivial homotopy}
Let $\psX$ be a product system over $P$ with coefficient algebra $A$.  For each $p \in P$, define
$Y_p := X_p \otimes C([0,1])$. Define multiplication on $Y = \{Y_p\}_{p \in P}$ by
\[
    (x \otimes f) (y \otimes g) := xy \otimes fg.
\]

We first claim that $Y$ is a homotopy of product systems from $X$ to $X$. It is easy to check that
$Y$ is a nondegenerate product system over $A \otimes C([0,1])$, using that the external tensor
product factors through the internal tensor product of Hilbert bimodules. It is then a homotopy of
product systems by definition, and it is standard that the quotient $X_p \otimes C([0,1])/(X_p
\otimes C_0([0,1] \setminus \{t\}) \cong X_p$, so $Y$ is a homotopy of product systems from $X$ to
$X$.

We now claim that
\[
    \mathcal{NO}_{Y} \cong \mathcal{NO}_\psX \otimes C([0,1]) \quad \text{ and } \quad \mathcal{NT}_{Y} \cong \mathcal{NT}_\psX \otimes C([0,1]).
\]
The universal property of $\mathcal{NO}_X$ gives a homomorphism $\rho: \mathcal{NO}_X \to
\mathcal{NO}_Y$ such that $\rho(j_p(x)) = j_p(x \otimes 1_{C([0,1])})$ for all $p \in P$ and $x
\in X$. It is routine to check that there is a homomorphim $\sigma : C([0,1]) \to
\mathcal{ZM}(\mathcal{NO}_Y)$ such that $\sigma(f)j_p(x \otimes g) = j_p(x \otimes fg)$. The
universal property of the tensor product then gives a homomorphism $(\rho \otimes \pi) :
\mathcal{NO}_X \otimes C([0,1]) \to \mathcal{NO}_Y$ such that $(\rho \otimes \pi)(j_p(x) \otimes
f) = \rho(j_p(x))\sigma(f) = j_p(x \otimes f)$.

For each $p$ the map $x \otimes f \mapsto j_p(x) \otimes f$ is a module map from $Y_p$ to
$\mathcal{NO}_X \otimes C([0,1])$, and routine calculations show that these assemble into a
representation $\tau$ of $Y$ in $\mathcal{NO}_X \otimes C([0,1])$. The modules $\widetilde{Y}_p$
invoked in the definition of Cuntz--Pimsner covariance are canonically isomorphic with the
corresponding modules $\widetilde{X}_p \otimes C([0,1])$. Likewise, each $\mathcal{K}(X_p \otimes
C([0,1]))$ is canonically isomorphic to $\mathcal{K}(X_p) \otimes C([0,1])$. A sum $\sum_{p \in F}
\widetilde{\iota}_p^{s}(T_p \otimes f_p)$ is zero in $\Ll(\widetilde{Y_s})$ for large $s$ if and
only if each $\sum_{p \in F} \iota_p^{s}(f_p(t)T_p)$ is eventually zero in $\Ll(\widetilde{X}_s)$,
and we deduce that $\tau$ is Cuntz--Nica--Pimsner covariant, and therefore induces a homomorphism
$\Pi \tau : \mathcal{NO}_Y \to \mathcal{NO}_X \otimes C([0,1])$. It is routine to see that $\Pi
\tau$ and $(\rho \otimes \pi)$ are mutually inverse, and hence isomorphisms.
\end{example}

\begin{prop}\label{prp:homotopy equiv rel}
Let $A$ be a $C^*$-algebra and let $P$ be a semigroup. Then homotopy of product systems is an
equivalence relation on the collection of product systems over $P$ of Hilbert bimodules over $A$.
\end{prop}
\begin{proof}
Example~\ref{eg:trivial homotopy} shows that the external tensor product $Y \otimes C([0,1])$ is a
homotopy from $Y$ to $Y$. let $F : [0,1] \to [0,1]$ be the flip homeomorphism $F(t) = 1-t$, and
let $X$ be a homotopy from $Y$ to $Z$. Then there is a product system $F_* X$ given by setting
$(F_* X)_p = X_p$ as a vector space, but with inner-product and actions given by $a \mathbin{:} x
= (a \circ F) \cdot x$, $x \mathbin{:} a = x\cdot (a \circ F)$ and $\llangle x,
y\rrangle_{C([0,1], A)} = \langle x, y\rangle_{C([0,1], A)} \circ F$. This $F_*(X)$ is a homotopy
from $Z$ to $Y$.

Finally, suppose that $W$ is a homotopy from $Y_1$ to $Y_2$, and $X$ is a homotopy from $Y_2$ to
$Y_3$. Then, by definition, there are isomorphisms of product systems $U : W^1 \to Y_2$ and $V :
X^0 \to Y_2$. Define subspaces $Z_p \le W_p \oplus X_p$ by
\[
Z_p := \{(w,x) \in W_p \oplus X_p : U(w + W_p\cdot I_1) = V(x + X_p \cdot I_0)\}.
\]
Routine calculations show that each $Z_p$ is invariant for the left and right actions of the
subalgebra $B := \{(a,b) \in C([0,1], A) \oplus C([0,1], A) : a(1) = b(0)\}$, and also that
$\langle Z_p, Z_p\rangle_{C([0,1], A) \oplus C([0,1], A)} \subseteq B$. It is straightforward to
check that the product-system structure on $W \oplus X$ restricts to give $Z$ the structure of a
nondegenerate product system over $P$ of Hilbert bimodules over $B$. We have an isomorphism
$\varphi : B \to C([0,1], A)$ given by
\[
\varphi(a,b)(t) = \begin{cases}
    a(2t)   &\text{ if $t \le 1/2$}\\
    b(2t-1) &\text{ if $t > 1/2$.}
    \end{cases}
\]
Using this to view $Z$ as a product system of Hilbert bimodules over $C([0,1], A)$, we see that $Z$
is fibred over $[0,1]$ and that $Z^0 = W^0 \cong Y_1$ while $Z^1 = X^1 \cong Y_3$; so $Z$ is a
homotopy from $Y_1$ to $Y_3$.
\end{proof}

We now begin exploring the relation between $\mathcal{NT}_X$ and $\mathcal{NT}_{X^t}$ for a
homotopy $X$ of product systems.  Our goal is to show the commutativity of the following diagram,
where $\varepsilon_t: X \to X^t$ is induced by the ``evaluation at $t$'' map  $C([0,1], A) \to A$
on the coefficient algebras.
\[
\begin{tikzpicture}
    \node (33) at (3,2) {$\mathcal{NT}_X$};
    \node (63) at (6,2) {$\mathcal{NO}_X$};
    \node (30) at (3,0) {$\mathcal{NT}_{X^t}$};
    \node (60) at (6,0) {$\mathcal{NO}_{X^t}$};
    \node (03) at (0,2) {$X$};
    \node (00) at (0,0) {$X^t$};
    \draw[-stealth] (03)-- node[below] {$i_X$} (33);
    \draw[-stealth] (00)-- node[above] {$i_{X^t}$} (30);
    \draw[-stealth] (30)-- node[above] {$q_t$} (60);
    \draw[-stealth] (33)-- node[below] {$q$} (63);
    \draw[-stealth] (03) -- node[left] {$\varepsilon_t$} (00);
    \draw[-stealth] (33)-- node[left] {$\varepsilon_{t*}$} (30);
    \draw[-stealth] (63)--node[right]{$\widetilde{\varepsilon}_{t*}$} (60);
    \draw[-stealth] (03) edge [out=30, in=150] node[above]{$j_X$} (63);
    \draw[-stealth] (00) edge [out=330, in=210] node[below]{$j_{X^t}$} (60);
\end{tikzpicture}
\]

\begin{prop}
\label{pr:evaluation-maps-commute} Let $A$ be a $C^*$-algebra, let $P$ be a quasi-lattice ordered
semigroup, and let $X$ be a nondegenerate, compactly aligned product system over $P$ of
Hilbert bimodules over $C([0,1], A)$. Suppose that $X$ is fibred over $[0,1]$.
\begin{enumerate}
\item There is a unique homomorphism $\iota : C([0,1]) \to \mathcal{ZM}(\mathcal{NT}_X)$ such
    that $\iota(f) i_e(g) = i_e(t \mapsto f(t)g(t))$ for all $g \in C([0,1], A)$, and this
    $\iota$ is unital.
\item For each $t \in [0,1]$, let $I^\mathcal{T}_t$ be the ideal $\{\iota(f) b : f \in C([0,1]),
    f(t) = 0, \text{ and } b \in \mathcal{NT}_X\}$. There is a surjective homomorphism
    $\varepsilon_{t*} : \mathcal{NT}_X \to \mathcal{NT}_{X^t}$ such that
    $\varepsilon_{t*}(i_p(x)) = i_p(x + X_p \cdot I_t)$ for all $p \in P$ and $x \in X_p$, and
    we have $\ker(\varepsilon_{t*}) = I^\mathcal{T}_t$.
\item Let $q : \mathcal{NT}_X \to \mathcal{NO_X}$ and $q_t : \mathcal{NT}_{X^t} \to
    \mathcal{NO}_{X^t}$ be the quotient maps. Then $\jmath := q \circ \iota$ is a unital
    homomorphism $\jmath : C([0,1]) \to \mathcal{ZM}(\mathcal{NO}_X)$, each $\varepsilon_{t*}$
    descends to a homormophism $\tilde\varepsilon_{t*} : \mathcal{NO}_X \to \mathcal{NO}_{X^t}$,
    and if we write $I^\mathcal{O}_t := q(I^{\mathcal{T}}_t)$ then we have $I^\mathcal{O}_t =
    \{\jmath(f) b : f \in C([0,1]), f(t) = 0, \text{ and } b \in \mathcal{NO}_X\}$.
\item Suppose that $P$ is directed, and that for each $p \in P$ the left action of $C([0,1], A)$
    on $X_p$ is injective and by compacts. Then $\ker(\tilde\varepsilon_{t*}) =
    I^\mathcal{O}_t$.
\end{enumerate}
\end{prop}
\begin{proof}
(1) Since $X$ is nondegenerate, any approximate identity $(a_i)$ for $C([0,1], A)$ satisfies $a_i
\cdot x, x \cdot a_i \to x$ for all $x \in X$. Hence $i_e(a_i)i_p(x), i_p(x) i_e(a_i) \to i_p(x)$
for all $p \in P$ and $x \in X_p$, and we deduce that $i_e : C([0,1], A) \to \mathcal{NT}_X$ is
nondegenerate, and so extends to a unital homomorphism $\tilde{i}_e : \mathcal{M}(C([0,1], A)) \to
\mathcal{M}(\mathcal{NT}_X)$. Since $\mathcal{M}(C([0,1], A)) = C([0,1], \mathcal{M}(A))$, this
$\tilde{i}_e$ restricts to a unital homomorphism $\iota : C([0,1]) \to
\mathcal{M}(\mathcal{NT}_X)$. The range of $\iota$ is central because for each $x \in X_p$ we have
$\iota(f)i_p(x) = i_p(f \cdot x) = i_p(x \cdot f) = i_p(x) \iota(f)$ because $X_p$ is fibred over
$[0,1]$. This homomorphism clearly satisfies $\iota(f) i_e(g) = i_e(t \mapsto f(t)g(t))$, and for
uniqueness, we observe that if $\iota' : C([0,1]) \to \mathcal{ZM}(\mathcal{NT}_X)$ satisfies the
same formula, then for any $x \in X$ we can use Cohen factorisation to write $x = y \cdot g$ for
some $g \in C([0,1], A)$, and then $\iota'(f)i_p(x) = \iota'(f)i_p(y)i_e(g) =
i_p(y)\iota'(f)i_e(g)$ by centrality, and since $\iota'(f)i_e(g) = i_e(fg)$ by assumption, we then
have $\iota'(f)i_p(x) = i_p(y)i_e(fg) = i_p(y \cdot (fg)) = i_p(x \cdot f)$, and the same
calculation applied to $\iota$ instead of $\iota'$ shows that this is also equal to
$\iota(f)i_p(x)$.  Since elements of the form $i_p(x) i_q(y)^*$ are dense in $\mathcal{NT}_X$, it follows that $\iota = \iota'$.

(2) We first show that $\{ \varepsilon_{t*} \circ i_p\}_{p \in P} $ is a Nica covariant
representation of $X$ in $\mathcal{NT}_{X^t}$. It is a representation because $i$ is a
representation; in particular, if $x \in X_p, y \in X_q$ we have
\[
(\varepsilon_{t*} \circ i_p(x)) (\varepsilon_{t*} \circ i_q(y))
	 = i_p(x + X_p \cdot I_t) i_q (y+ X_q \cdot I_t)
     = i_{pq}(xy + X_{pq} \cdot I_t)
     = \varepsilon_{t*}(i_{pq}(xy)),
\]
and (for $x, y \in X_p$)
\begin{align*}
i_e(\langle x, y \rangle_{A \otimes C([0,1])} + X_e \cdot I_t)
 & =   i_p(x + X_p \cdot I_t)^* i_p(y + X_p \cdot I_t)
    = (\varepsilon_{t*} \circ i_p(x))^* (\varepsilon_{t*} \circ i_p(y)) \\
  & = \varepsilon_{t*} \circ i_e( \langle x, y\rangle_{A \otimes C([0,1])}.
\end{align*}
To see that it is Nica covariant, first note that for $x,y \in X_p$, Equation~\eqref{eq:induced
composition} shows that $(\varepsilon_{t*} \circ i)^{(p)} = \varepsilon_{t*} \circ i^{(p)}$. Now
fix $S \in \mathcal{K}(X_p)$ and $T \in \mathcal{K}(X_q)$. Then
\[(\varepsilon_{t*} \circ i)^{(p)}(S)
(\varepsilon_{t*} \circ i)^{(q)}(T) = \varepsilon_{t*}(i^{(p)}(S)i^{(q)}(T))\] and
$(\varepsilon_{t*} \circ i)^{(p \vee q)}(\iota_p^{p \vee q}(S)\iota_q^{p \vee q}(T)) =
\varepsilon_{t*}(i^{(p \vee q)}(\iota_p^{p \vee q}(S)\iota_q^{p \vee q}(T)))$, and so
$\varepsilon_{t*} \circ i$ is Nica covariant because $i$ is.

The universal property of $\mathcal{NT}_X$ now yields a homomorphism $\varepsilon_{t*}: \mathcal{NT}_X \to
\mathcal{NT}_{X^t}$. This homomorphism is surjective because each $i_p(x + X_p \cdot I_t) =
\varepsilon_{t*}(i_p(x))$ and each $i_e(a) = \varepsilon_{t*}(i_e(a \otimes 1_{C([0,1])}))$
belongs to its range.

To see that $\ker(\varepsilon_{t*}) = I^{\mathcal T}_t$, we first observe that for any $f \in
C_0([0,1]\setminus\{t\}), x \in X_p$,
\[ \varepsilon_{t*}(\iota(f) i_p(x)) = \varepsilon_{t*} (i_p((f \otimes 1_A) \cdot x) = i_p( x \cdot (f \otimes 1_A) + X_p \cdot I_t) = 0,\]
since $x \cdot (f \otimes 1_A) \in X_p \cdot I_t$.  Consequently, $I^{\mathcal T}_t \subseteq \ker \varepsilon_{t*}$.  To establish the other containment, we
 will construct a Nica-covariant
representation $\psi$ of $X^t$ in $\mathcal{NT}_X/I^{\mathcal T}_t$, such that the resulting
homomorphism $\Pi\psi : \mathcal{NT}_{X^t} \to \mathcal{NT}_X/I^{\mathcal T}_t$ satisfies
$  \varepsilon_{t*} \circ \Pi\psi  = \operatorname{id}$. We aim to define
\[
    \psi(i_p(x + X_p \cdot I_t)) = i_p(x) + I^{ \mathcal T}_t.
\]
To see that this is well defined, observe that for all $x,y \in X_p$, we have
\begin{align*}
(i_p(x) + I^{\mathcal T}_t)^*(i_p(y) + I^{\mathcal T}_t)
    &= i_p(x)^*i_p(y) + I^{ \mathcal T}_t
    = i_e(\langle x, y\rangle_{A \otimes C([0,1])}) + I^{\mathcal{T}}_t\\
    &= \langle i_p(x + X_p \cdot I_t), i_p(y + X_p \cdot I_t)\rangle_{A}.
\end{align*}
Hence the formula for $\psi$ determines an isometric map from $X^t_p$ to $\mathcal{NT}_X / I^{
\mathcal T}_t$. An argument like that given for $\varepsilon_{t*} \circ i$ above shows that $\psi$
is a Nica-covariant representation. To see that $\varepsilon_{t*} \circ \psi = \operatorname{id}$,
we compute:
\[
    \varepsilon_{t*}(\psi(i_p(x+ X_p \cdot I_t)) = \varepsilon_{t*}(i_p(x) + I_t^{\mathcal T}) = \varepsilon_{t*}(i_p(x)) = i_p(x+ X_p \cdot I_t).
\]

(3) To see that $\varepsilon_{t*}$ descends to a homomorphism $\tilde \varepsilon_{t*} :
\mathcal{NO}_X \to \mathcal{NO}_{X^t}$, it suffices to show that if $q_t : \mathcal{NT}_{X^t} \to
\mathcal{NO}_{X^t}$ is the quotient map, then $q_t \circ \varepsilon_{t*} \circ i_p$ is a
Cuntz--Pimsner covariant representation of $X$. Since the actions of $A \otimes C([0,1])$ on the
modules $X_p$ are implemented by injective homomorphisms, the modules $\widetilde{X}_p$ of
\cite{sims--yeend} are just the original modules $X_p$, and the homomorphisms
$\tilde{\iota}^{pq}_p : \Ll(X_p) \to \Ll(\widetilde{X}_{pq})$ are the standard inclusions
$\iota^{pq}_p : \Ll(X_p) \to \Ll(X_{pq})$ characterised by $\iota_p^{pq}(S)(xy) = (Sx)y$ for $x
\in X_p$ and $y \in X_q$. Write $\iota[t]^{pq}_p : \Ll(X^t_p) \to \Ll(X^t_{pq})$ for the
corresponding inclusions for the product system $X^t$; and as usual for $T \in \Ll(X_p)$ write
$T_t$ for the map in $\Ll(X^t_p)$ given by $T_t(x + X \cdot I_t) = Tx + X \cdot I_t$. A routine
calculation shows that for $T \in \Ll(X_p)$ we have $\iota[t]^{pq}_p(T_t) = \iota^{pq}_p(T)_t$. So
if $F \subseteq P$ is finite and $T_p \in \mathcal{K}(X_p)$ for each $p \in F$, then
\begin{equation}\label{eq:shift the ts}
\Big(\sum_{p \in F} \iota^s_p(T_p)\Big)_t = \sum_{p \in F} \iota[t]^s_p((T_p)_t)\qquad\text{ for every $s \ge
p$ in $P$.}
\end{equation}
Assume $\sum_{p \in F} \iota^s_p(T_p) = 0$ for large $s$; we will show that $\sum_{p \in F} (q_t
\circ \varepsilon_{t*} \circ i)^{(p)}(T_p) = 0$. From~\eqref{eq:shift the ts} we see that $\sum_{p
\in F} \iota[t]^s_p((T_p)_t) = 0$ for large $s$, and consequently
\[
q_t\Big(\sum_{p \in F} i^{(p)}_t((T_p)_t)\Big)
    = q_t\Big( \sum_{p \in F} \iota_p^s[t]((T_p)_t)\Big)= 0.
\]
Equation~\ref{eq:induced composition} shows that  $(q_t \circ \varepsilon_{t*} \circ i)^{(p)} =
(q_t \circ \varepsilon_{t*}) \circ i^{(p)}$. Hence
\[
\sum_{p \in F} (q_t \circ \varepsilon_{t*} \circ i)^{(p)}((T_p)_t)
    = \sum_{p \in F} (q_t \circ \varepsilon_{t*}) \circ i^{(p)}(T_p) = 0.
\]
That $I^{\mathcal{O}}_t = \{\jmath(f) b : f \in C([0,1]), f(t) = 0, \text{ and } b \in
\mathcal{NO}_X\}$ follows from the definition of $I^{\mathcal{T}}_t$.

(4) Since $I_t^{\mathcal O} = q(I_t^{\mathcal T})= q(\ker(\varepsilon_{t*}))$, we clearly have
$I_t^{\mathcal O} \subseteq \ker (\tilde \varepsilon_{t*})$. To see that $\ker(\tilde
\varepsilon_{t*})= q(I_t^{\mathcal T})$, it suffices to show that if $q_t :
\mathcal{NT}_X/I^{\mathcal{T}}_t \to \mathcal{NO}_X/I^{\mathcal{O}}_t$ is the quotient map, then
the Nica-covariant representation $\psi : X^t \to \mathcal{NT}_X/I^{\mathcal T}_t$ defined above
has the property that $q_t \circ \psi$ is Cuntz--Nica--Pimsner covariant. For each $p \in P$, let
$\phi^t_p : A \to \mathcal{K}(X_p^t)$ be the injective homomorphism that implements the left
action; and let $\phi_p : A \otimes C([0,1]) \to \mathcal{K}(X_p)$ be the corresponding map for
$X$. We invoke \cite[Proposition~5.1(2)]{sims--yeend}, which says that it suffices to show that
for each $a \in A$ and each $p \in P$, we have $(q_t \circ \psi)^{(p)}(\phi^t_p(a)) - (q_t \circ
\psi_e)(a) = 0$. By \cite[Proposition~5.1(2)]{sims--yeend}, we have $i^{(p)}(\phi_p(a \otimes
1_{[0,1]})) - i_e(a \otimes 1_{[0,1]}) \in \ker(q)$. Since the identification of $A \otimes
C([0,1]))/A \otimes C_0([0,1] \setminus\{t\})$ with $A$ carries $a \otimes 1_{[0,1]}$ to $a$, we
have
\begin{align*}
(q_t \circ \psi)^{(p)}(\phi^t_p(a)) - (q_t \circ \psi_e)(a)
    &= q_t(\psi^{(p)}(\phi^t_p(a)) - \psi_e(a))\\
    &= q_t\big(i^{(p)}(\phi_p(a \otimes 1_{[0,1]})) - i_e(a \otimes 1_{[0,1]}) + I^{\mathcal{T}}_t\big)\\
    &= q\big(i^{(p)}(\phi_p(a \otimes 1_{[0,1]})) - i_e(a \otimes 1_{[0,1]})\big) + I^{\mathcal{O}}_t
     = 0.\qedhere
\end{align*}
\end{proof}

\begin{rmk}
It is unclear whether the hypothesis that $p \vee q < \infty$ for all $p,q \in P$ is necessary for
the final assertion of Proposition~\ref{pr:evaluation-maps-commute}(3). The issue is that to
establish the final assertion without this hypothesis, we would need to verify the Cuntz--Pimsner
relation for the map $q \circ \psi$ in the final paragraph of the proof. That is, given a finite
$F \subseteq P$ and elements $T^t_p \in \mathcal{K}(X^t_p)$ such that $\sum_p \iota[t]^s_p(T^t_p)
= 0$ for large $s$, we would need to show that we could find representatives $T_p \in
\mathcal{K}(X_p)$ of the operators $T^t_p$ with the property that $\sum_p \iota^s_p(T_p) = 0$ for
large $s$; it is not clear to us whether this is so.
\end{rmk}

We now prove that the evaluation maps $\varepsilon_t$ for a homotopy of product
systems over $\N^k$ induce $KK$-equivalences of Nica--Toeplitz algebras.

\begin{cor}
Let $X$ be a homotopy of product systems over $\N^k$ with coefficient algebra $A$. Suppose that
$A$ is $\sigma$-unital and each $X_n$ is countably generated. For any $t \in [0,1]$, the class of
the homomorphism $\varepsilon_{t*}$ of Proposition~\ref{pr:evaluation-maps-commute}(2) is a
$KK$-equivalence from $\mathcal{NT}_X$ to $\mathcal{NT}_{X^t}$, and in particular defines an
isomorphism
\[
    K_*(\mathcal{NT}_{X}) \cong K_*(\mathcal{NT}_{X^t}).
\]
\end{cor}
\begin{proof}
By \cite[Corollary 4.18]{fletcher-NYJM}, the classes of $i: C([0,1], A) \to \mathcal{NT}_X$ and $i^t : A
\to \mathcal{NT}_{X^t}$ are $KK$-equivalences. Since $[0,1]$ is contractible, the evaluation map
$\varepsilon_t^0 : C([0,1], A) \to A$ also induces a $KK$-equivalence. We have $i^t \circ
\varepsilon^0_t = \varepsilon_{t*} \circ i$, and so in $KK$-theory,  we have
\[
[\varepsilon_{t*}] = [i^t] \hatimes [\varepsilon^0_t] \hatimes [i]^{-1}.
\]
Since the three factors on the right-hand side are $KK$-equivalences, we deduce that
$[\varepsilon_{t*}]$ is a $KK$-equivalence as well.
\end{proof}

Our main theorem states that homotopy of product systems over $\N^k$ preserves the $K$-theory of
their Cuntz--Nica--Pimsner algebras.

\begin{thm}
\label{thm:htopy-KK} Let $X$ be a homotopy of product systems over $\N^k$ whose coefficient
algebra $A\otimes C([0,1])$ acts faithfully by compact operators. For any $t \in [0,1]$, the
evaluation map $\varepsilon_t$ induces an isomorphism
\[
    K_*(\mathcal{NO}_{X^t}) \cong K_*(\mathcal{NO}_{X}).
\]
In particular, if $A$ satisfies the UCT and each $X_n$ is countably generated, then each $\mathcal{NO}_{X^t}$ is $KK$-equivalent to
$\mathcal{NO}_{X}$.
\end{thm}
\begin{proof}
We will proceed by induction on $k$. Recall that $X_0 \cong C([0,1],A)$, and so $X_0^t \cong A$.
Hence, if $k = 0$ then $\mathcal{NO}_X = C([0,1], A)$ and $\mathcal{NO}_{X^t} = A$, and
$\widetilde{\varepsilon_{t*}}$ is just the evaluation map $\varepsilon_t : C([0,1], A) \to A$.
Thus $\widetilde{\varepsilon_{t*}}$ induces a $KK$-equivalence.

Now suppose that the result holds for all product systems over $\N^{k-1}$. Let $X_{<k}$ be the
product system over $\N^{k-1}$ obtained by restriction of $X$, which is also a homotopy of product
systems with coefficient algebra $A$. Theorem~4.7 of \cite{fletcher-NYJM} (see also
\cite[Theorem~3.4.21]{fletcher-thesis-1}) shows that there is a product system $Y$ of Hilbert
$\mathcal{NO}_{X_{<k}}$-bimodules over $\N$ in which the left action is faithful and by compact
operators and such that $\mathcal{NO}_X \cong \mathcal{NO}_Y$. As $Y$ is a product system over
$\N$,  $\mathcal{NO}_Y = \mathcal{O}_{Y_1}$.

By definition (see \cite[Proposition~4.3]{fletcher-NYJM} or
\cite[Proposition~3.4.7]{fletcher-thesis-1}), if $\jmath : \mathcal{NO}_{X_{<k}} \to
\mathcal{NO}_{X}$ is the homomorphism induced by the inclusion $X_{<k} \hookrightarrow X$, then
the module $Y_1$ is equal to $\clsp\{i_{e_k}(\xi)\jmath(b) : \xi \in X_{e_k}, b \in
\mathcal{NO}_{X_{<k}}\} \subseteq \mathcal{NO}_X$, and similarly for $Y^t_1$ for each fixed $t$.
So the map $\widetilde{\varepsilon_{t*} }: \mathcal{NO}_X \to \mathcal{NO}_{X^t}$ restricts to a bimodule map
$(\varepsilon^Y_{t}, {\varepsilon^{<k}_{t*}}) : (Y_1, \mathcal{NO}_{X_{<k}}) \to (Y_1^t,
\mathcal{NO}_{X^t_{<k}})$, and by definition the homomorphism $\mathcal{NO}_X \to
\mathcal{NO}_{X^t}$ induced by this bimodule map is the original $\widetilde{\varepsilon_{t*}}$. Hence the
final statement (concerning naturality) of \cite[Theorem~4.4]{KPS7}, shows that Pimsner's six-term
exact sequences in $KK$-theory \cite[Theorem~4.9(1)]{Pimsner} applied with $B = \C$ to each of
$Y_1$ and $Y^t_1$, assemble into the following commuting diagram:
\[
\begin{tikzpicture}[yscale=1.2, xscale=1.4]
 \node[inner sep= 1pt] (I-1) at (4.5,-2.5) {$K_1(\mathcal{NO}_{X_{<k}})$.};
    \node[inner sep= 1pt] (A-1) at (0, -2.5) {$K_1(\mathcal{NO}_{X_{<k}})$};
    \node[inner sep= 1pt] (quo-1) at (-4.5,-2.5) {$K_1(\mathcal{NO}_{X})$};
    \node[inner sep= 1pt] (quo-0) at (4.5,2.5) {$K_0(\mathcal{NO}_{X})$};
    \node[inner sep= 1pt] (A-0) at (0,2.5) {$K_0(\mathcal{NO}_{X_{<k}})$};
    \node[inner sep= 1pt] (I-0) at (-4.5,2.5) {$K_0(\mathcal{NO}_{X_{<k}})$};
\draw[->] (I-1) to node[pos=0.5, above] {$1 - [Y_1]$} (A-1);
\draw[->] (A-1) to node[pos=0.5, above] {$[i_0]$} (quo-1) {};
\draw[->] (quo-1) to (I-0) {};
\draw[->] (I-0) to node[pos=0.5, above] {$1 - [Y_1]$} (A-0);
\draw[->] (A-0) to node[pos=0.5, above] {$[i_0]$} (quo-0) {};
\draw[->] (quo-0) to (I-1) {};

    \node[inner sep= 1pt] (I-1-in) at (3,-1) {$K_1(\mathcal{NO}_{X^t_{<k}})$};
    \node[inner sep= 1pt] (A-1-in) at (0, -1) {$K_1(\mathcal{NO}_{X^t_{<k}})$};
    \node[inner sep= 1pt] (quo-1-in) at (-3,-1) {$K_1(\mathcal{NO}_{X^t} )$};
    \node[inner sep= 1pt] (quo-0-in) at (3,1) {$K_0(\mathcal{NO}_{X^t} )$};
    \node[inner sep= 1pt] (A-0-in) at (0,1) {$K_0(\mathcal{NO}_{X^t_{<k}})$};
    \node[inner sep= 1pt] (I-0-in) at (-3,1) {$K_0(\mathcal{NO}_{X^t_{<k}})$};
\draw[->] (I-1-in) to node[pos=0.5, above] {$1 - [Y^t_1]$} (A-1-in) {};
\draw[->] (A-1-in) to node[pos=0.5, above] {$[i^t_0]$} (quo-1-in) {};
\draw[->] (quo-1-in) to (I-0-in) {};
\draw[->] (I-0-in) to node[pos=0.5, above] {$1 - [Y^t_1]$} (A-0-in) {};
\draw[->] (A-0-in) to node[pos=0.5, above] {$[i^t_0]$} (quo-0-in) {};
\draw[->] (quo-0-in) to (I-1-in) {};

\draw[->] (I-1) to node[pos=0.75, anchor=north east, inner sep=0.5pt] {$[\varepsilon^{<k}_{t*}]$} (I-1-in) {};
\draw[->] (A-1) to node[pos=0.5, anchor=east, inner sep=1pt] {$[\varepsilon^{<k}_{t*}]$} (A-1-in) {};
\draw[->] (quo-1) to node[pos=0.5, anchor=south east, inner sep=0.5pt] {$[\varepsilon_{t*}]$} (quo-1-in) {};
\draw[->] (I-0) to node[pos=0.5, anchor=south west, inner sep=0.5pt] {$[\widetilde{\varepsilon^{<k}_{t*}}]$} (I-0-in) {};
\draw[->] (A-0) to node[pos=0.5, anchor=west, inner sep=1pt] {$[\varepsilon^{<k}_{t*}]$} (A-0-in) {};
\draw[->] (quo-0) to node[pos=0.5, anchor=north west, inner sep=0.5pt] {$[\widetilde{\varepsilon_{t*}}]$} (quo-0-in) {};
\end{tikzpicture}
\]

The inductive hypothesis implies that the upper-left, upper-middle, lower-right and lower-middle
outside-to-inside maps are isomorphisms, and so the Five Lemma shows that the remaining two
outside-to-inside maps are isomorphisms as well.

Corollary 5.21 of \cite{fletcher-NYJM} establishes that if $A$ satisfies the Universal Coefficient
Theorem, then our hypotheses on $A$ and $X$ imply that $\mathcal{NO}_X$ and $\mathcal{NO}_{X^t}$
also satisfy the UCT.  Consequently, the above isomorphisms in $K$-theory give a $KK$-equivalence
in this case.
\end{proof}

\section{Applications to \texorpdfstring{$k$}{k}-graphs}\label{sec:k-graphs}

The main result of this section, Theorem~\ref{thm:twisted 2-graph K-th}, uses
Theorem~\ref{thm:htopy-KK} to show that any two twisted Cuntz--Krieger algebras of $2$-graphs with
isomorphic skeletons have isomorphic $K$-groups. Our approach relies on casting the notion of
homotopy of 2-cocycles for $k$-graphs (previously studied in relation to $K$-theory in \cite{KPS5,
gillaspy-PJM}) in the language of product systems. Most of the section is aimed at proving this
result about 2-graphs, but we also pause to investigate some potential applications to $k$-graphs
for $k \ge 2$. Consequently, many of the results that contribute to the proof of
Theorem~\ref{thm:twisted 2-graph K-th} are stated for arbitrary $k$.

We begin by introducing the notation and definitions needed for our construction of the product systems associated to twisted $k$-graph $C^*$-algebras.

\subsection{Background and notation for twisted \texorpdfstring{$k$}{k}-graph \texorpdfstring{$C^*$}{C*}-algebras}

We write $\N^k$ for the monoid of $k$-tuples of nonnegative integers under coordinatewise
addition, and sometimes regard it as a small category with one object. We write $\{e_1, \dots,
e_k\}$ for the canonical generators of $\N^k$. Recall from \cite{kp2} that a {\em $k$-graph}, or a higher-rank graph of rank $k$, is a
countable small category $\Lambda$ equipped with a functor $d : \Lambda \to \N^k$ that satisfies
the {\em factorisation property}: whenever $d(\lambda) = m+n$ there exist unique $\mu \in d^{-1}(m)$ and
$\nu \in d^{-1}(n)$ such that $\lambda = \mu\nu$. We write $\Lambda^n := d^{-1}(n)$. The
factorisation property guarantees that $\Lambda^0 = \{\operatorname{id}_o : o \in
\operatorname{Obj}(\Lambda)\}$. Given subsets $E,F \subseteq \Lambda$, we write $EF = \{\mu\nu :
\mu \in E, \nu \in F, s(\mu) = r(\nu)\}$. In the instance of singleton sets, we simplify notation
by writing $\mu F$ rather than $\{\mu\} F$ and $E \nu$ rather than $E \{\nu\}$.

We say that the $k$-graph $\Lambda$ is \emph{row finite with no sources} if $0 < |v\Lambda^n| <
\infty$ for all $v \in \Lambda^0$ and $n \in \N^k$.  This hypothesis, which we will assume
throughout this paper, facilitates the construction and analysis of the  $C^*$-algebra associated
to the $k$-graph.

As introduced in \cite{kps3, KPS4}, there are two types of 2-cocycle on a $k$-graph $\Lambda$:
categorical and cubical. However, \cite[Theorem~4.15]{KPS4} exhibits an isomorphism between the
cubical and categorical second cohomology groups (see also \cite{GW}).  Moreover, by
\cite[Corollary~5.7]{KPS4}, this isomorphism induces an isomorphism of the associated twisted
$C^*$-algebras, so one may freely choose to work with categorical or cubical 2-cocycles without
loss of generality. We have found cubical 2-cocycles to be a more natural framework for describing
the product systems associated to $k$-graphs.

To define cubical cocycles, observe that if $1 \le i_1 < i_2 < i_3 \le k$, and if $d(\lambda) =
e_{i_1} + e_{i_2} + e_{i_3}$, then the factorisation property shows that for each $j \in \{1, 2,
3\}$ there are factorisations $\lambda = F^0_j(\lambda)\beta_j = \alpha_j F^1_j(\lambda)$ such
that $d(\alpha_j) = d(\beta_j) = e_{i_j}$.
\[
\begin{tikzpicture}[xscale=2.2,yscale=2]
\begin{scope}
    \node (flo) at (0,0,1) {};
    \node (flt) at (0,1,1) {};
    \node (fro) at (1,0,1) {};
    \node (frt) at (1,1,1) {};
    \node (alo) at (0,0,0) {};
    \node (alt) at (0,1,0) {};
    \node (aro) at (1,0,0) {};
    \node (art) at (1,1,0) {};
    \node at (1, 0.5, 0.5) {\tiny$F^1_1 (\lambda)$};
    \draw[->] (fro) to node[pos=0.5, above] {$\alpha_1$} (flo) {};
    \draw[->, dashed] (art) to (aro) {};
    \draw[->, dashed] (frt) to (fro) {};
    \draw[->, dotted] (art) to (frt) {};
    \draw[->, dotted] (aro) to (fro) {};
\end{scope}
\begin{scope}[yshift=-1.7cm]
    \node (flo) at (0,0,1) {};
    \node (flt) at (0,1,1) {};
    \node (fro) at (1,0,1) {};
    \node (frt) at (1,1,1) {};
    \node (alo) at (0,0,0) {};
    \node (alt) at (0,1,0) {};
    \node (aro) at (1,0,0) {};
    \node (art) at (1,1,0) {};
    \node at (0, 0.5, 0.5) {\tiny$F^0_1(\lambda)$};
    \draw[->] (art) to node[pos=0.5, above] {$\beta_1$} (alt) {};
    \draw[->, dashed] (alt) to (alo) {};
    \draw[->, dashed] (flt) to (flo) {};
    \draw[->, dotted] (alt) to (flt) {};
    \draw[->, dotted] (alo) to (flo) {};
\end{scope}
\begin{scope}[xshift=2.2cm]
    \node (flo) at (0,0,1) {};
    \node (flt) at (0,1,1) {};
    \node (fro) at (1,0,1) {};
    \node (frt) at (1,1,1) {};
    \node (alo) at (0,0,0) {};
    \node (alt) at (0,1,0) {};
    \node (aro) at (1,0,0) {};
    \node (art) at (1,1,0) {};
    \node at (0.5, 0, 0.5) {\tiny$F^0_2(\lambda)$};
    \draw[->] (aro) to (alo) {};
    \draw[->] (fro) to (flo) {};
    \draw[->, dotted] (alo) to (flo) {};
    \draw[->, dotted] (aro) to (fro) {};
    \draw[->, dashed] (art) to node[pos=0.5, left] {$\beta_2$}  (aro) {};
\end{scope}
\begin{scope}[xshift=2.2cm, yshift=-1.7cm]
    \node (flo) at (0,0,1) {};
    \node (flt) at (0,1,1) {};
    \node (fro) at (1,0,1) {};
    \node (frt) at (1,1,1) {};
    \node (alo) at (0,0,0) {};
    \node (alt) at (0,1,0) {};
    \node (aro) at (1,0,0) {};
    \node (art) at (1,1,0) {};
    \node at (0.5, 1, 0.5) {\tiny$F^1_2(\lambda)$};
    \draw[->, dashed] (flt) to node[pos=0.5, left] {$\alpha_2$} (flo) {};
    \draw[->, dotted] (alt) to (flt) {};
    \draw[->, dotted] (art) to (frt) {};
    \draw[->] (art) to (alt) {};
    \draw[->] (frt) to (flt) {};
\end{scope}
\begin{scope}[xshift=4.4cm]
    \node (flo) at (0,0,1) {};
    \node (flt) at (0,1,1) {};
    \node (fro) at (1,0,1) {};
    \node (frt) at (1,1,1) {};
    \node (alo) at (0,0,0) {};
    \node (alt) at (0,1,0) {};
    \node (aro) at (1,0,0) {};
    \node (art) at (1,1,0) {};
    \node at (0.5, 0.5, 0) {\tiny$F^1_3(\lambda)$};
    \draw[->] (aro) to (alo) {};
    \draw[->] (art) to (alt) {};
    \draw[->, dashed] (art) to (aro) {};
    \draw[->, dashed] (alt) to (alo) {};
    \draw[->, dotted] (alo) to node[pos=0.5, anchor=south east] {$\alpha_3$}  (flo) {};
\end{scope}
\begin{scope}[xshift=4.4cm, yshift=-1.7cm]
    \node (flo) at (0,0,1) {};
    \node (flt) at (0,1,1) {};
    \node (fro) at (1,0,1) {};
    \node (frt) at (1,1,1) {};
    \node (alo) at (0,0,0) {};
    \node (alt) at (0,1,0) {};
    \node (aro) at (1,0,0) {};
    \node (art) at (1,1,0) {};
    \node at (0.5, 0.5, 1) {\tiny$F^0_3(\lambda)$};
    \draw[->] (fro) to (flo) {};
    \draw[->] (frt) to (flt) {};
    \draw[->, dashed] (frt) to (fro) {};
    \draw[->, dashed] (flt) to (flo) {};
    \draw[->, dotted] (art) to node[pos=0.5, anchor=south east] {$\beta_3$}  (frt) {};
\end{scope}
\end{tikzpicture}
\]
A \emph{cubical 2-cocycle} on a $k$-graph $\Lambda$ is a function $\phi :
\bigsqcup_{i\not=j}\Lambda^{e_i + e_j} \to \T$ that satisfies the following cocycle identity:
whenever $1 \le i_1 < i_2 < i_3 \le k$ and $d(\lambda) = e_{i_1} + e_{i_2} + e_{i_3}$ as above, we
have
\[
\phi(F^0_1(\lambda))\phi(F^1_2(\lambda))\phi(F^0_3(\lambda))
    = \phi(F^1_1(\lambda))\phi(F^0_2(\lambda))\phi(F^1_3(\lambda)).
\]
Intuitively, the products of the values over the opposing hemispheres centred on the
top-front-left vertex (the bottom row above) and the bottom-back-right vertex (the top row above)
of a cube are equal.

If $\phi$ is a cubical $2$-cocycle on a row-finite $k$-graph $\Lambda$ with no sources, then a {\em
Cuntz--Krieger $\phi$-representation} of $\Lambda$ in a $C^*$-algebra consists of mutually
orthogonal projections $\{p_v : v \in \Lambda^0\}$ and partial isometries $\{s_\mu : \mu \in
\bigsqcup_i \Lambda^{e_i}\}$ satisfying
\begin{enumerate}
\item $s^*_\mu s_\mu = p_{s(\mu)}$ for all $\mu \in \bigsqcup_i \Lambda^{e_i}$,
\item $s_{\mu'} s_{\nu'} = \phi(\mu\nu)s_\mu s_\nu$ whenever $1 \le i < j \le k$ and $\mu,\mu'
    \in \Lambda^{e_i}$ and $\nu,\nu' \in \Lambda^{e_j}$ satisfy $\mu\nu = \nu'\mu'$, and
\item $p_v = \sum_{\mu \in v\Lambda^{e_i}} s_\mu s^*_\mu$ for all $v \in \Lambda^0$ and $i \le
    k$.
\end{enumerate}
The twisted $C^*$-algebra $C^*_\phi(\Lambda)$ is the universal $C^*$-algebra generated by a
Cuntz--Krieger $\phi$-representation of $\Lambda$. If $\phi$ is the trivial 2-cocycle, so that
$\phi(\lambda) = 1$ for all $\lambda \in \bigsqcup_{i \not= j} \Lambda^{e_i + e_j}$, then
$C^*_\phi(\Lambda) \cong C^*(\Lambda)$ is the usual $k$-graph $C^*$-algebra.

The following notions, while not standard in the literature on $k$-graphs, underlie our results in this section.
Given directed graphs $E, F$ with common vertex set $E^0 = F^0$, and given $v,w \in E^0$, we write
\[
E^1F^1 := \{ef : e \in E^1, f\in F^1, \text{ and }s(e) = r(f)\},
\]
and  for $v,w \in E^0$ we write
\[
v E^1 F^1 w = \{ef \in E^1F^1 : r(e) = v\text{ and }s(f) = w\}.
\]

By a \emph{$k$-skeleton}, we mean a tuple $E = (E_1, \dots, E_k)$ of row-finite directed graphs
with no sources and with common vertex set $E^0$ such that for all $v,w\in E^0$ and all $i \not= j
\le k$, we have
\[
|v E^1_i E^1_j w| = |vE^1_jE^1_i w|.
\]
We say that two $k$-skeletons $(E_1, \ldots, E_k)$ and $(F_1, \ldots, F_k)$ are {\em isomorphic}
if there exist bijections $\rho^0: E_0 \to F_0$ and $\rho^{i}: E_i \to F_i$ such that $\rho^0
\circ s = s \circ \rho^i, \rho^0 \circ r = r \circ \rho^i$ for $1 \leq i \leq k$.

Every $k$-graph $\Lambda$ gives rise to a $k$-skeleton (namely $E_i = \Lambda^{e_i}$).  The
converse is not true; however, the additional structure needed for a  $k$-skeleton  to give rise
to a $k$-graph is described in \cite{HRSW}.

\subsection{Unitary cocycles and higher-rank graphs}
The main result of this section is the following.

\begin{thm}\label{thm:twisted 2-graph K-th}
Suppose that $\Lambda$ and $\Gamma$ are row-finite $2$-graphs with isomorphic skeletons.  Let
$\phi : \Lambda^{(1,1)} \to \T$ and $\psi : \Gamma^{(1,1)} \to \T$ be cubical $2$-cocycles. Then
$K_*(C^*_\phi(\Lambda)) \cong K_*(C^*_\psi(\Gamma))$.
\end{thm}

The next lemma underlies Definition~\ref{def:unitary cocycle} below.

\begin{lemma}\label{lem:U's induce iso}
Let $E = (E_1, E_2)$ be a $2$-skeleton. For each $v,w \in E^0$ suppose that $U(v,w)$ is a unitary
operator $U(v,w) \in \mathcal{U}(\C^{v E^1_2 E^1_1 w}, \C^{v E^1_1 E^1_2 w})$. Let $A :=
C_0(E^0)$, and for $i = 1,2$, let $X_i := X(E_i)$ be the graph bimodule of $E_i$ as in
Example~\ref{ex:graph module}. Then there is an isomorphism $ U : X_2 \otimes_A X_1 \to X_1
\otimes_A X_2$ such that $ U(\delta_f \otimes \delta_e) = U(r(f),s(e))(\delta_f \otimes \delta_e)$
for all $fe \in E^1_2E^1_2$.
\end{lemma}
\begin{proof}
Let $X_{12}$ and $X_{21}$ be the graph modules for the directed graphs $(E^0, E^1_1E^1_2, r, s)$
and $(E^0, E^1_2E^1_1, r, s)$ respectively. Routine calculations with inner products (see the
proof of \cite[Proposition~3.2]{RaeS}) show that $X_{ij} \cong X_i \otimes_A X_j$ via a map
satisfying $\delta_{ef} \mapsto \delta_e \otimes \delta_f$, so it suffices to show that there is
an isomorphism $X_{21} \to X_{12}$ satisfying $\delta_{fe} \mapsto U(r(f), s(e))\delta_{fe}$ for
all $fe \in E^1_2 E^1_1$. For this, fix $fe, hg \in E^1_2 E^1_1$. For $w \in E^0$, we calculate:
\begin{align*}
\langle U(r(f), s(e)) &\delta_{fe}, U(r(h), s(g)) \delta_{hg}\rangle_A(w)\\
    &= \sum_{\beta \in E^1_1E^2_1w} \overline{(U(r(f), s(e))\delta_{fe})(\beta)} (U(r(h), s(g))\delta_{hg})(\beta).
\intertext{This is equal to zero unless $r(h) = r(f)$ and $s(g) = s(e)$, in which case, $U(r(h), s(g)) =
U(r(f), s(e))$, and since this is a unitary operator, we may continue the calculation}
    &= \delta_{r(f), r(h)}\delta_{s(e), s(g)} \langle U(r(f), s(e))\delta_{fe}, U(r(f), s(e)) \delta_{hg}\rangle_A(w)\\
    &= \langle \delta_{fe}, \delta_{hg}\rangle_A(w).
\end{align*}
It follows by linearity and continuity that there is an inner-product-preserving linear operator
\begin{equation}\label{eq:U's assemble}
U : X_{21} \to X_{12}
\end{equation}
such that $U|_{\delta_v \cdot X_{21} \cdot \delta_w} = U(v,w)$ for all $v,w$. Since each
$U|_{\delta_v \cdot X_{21} \cdot \delta_w}$ is surjective onto $\delta_v \cdot X_{12} \cdot
\delta_w$, and since these subspaces span $X_{12}$, we see that $U$ is surjective and hence an
isomorphism.
\end{proof}

\begin{defn}
\label{def:unitary cocycle}
Let $E = (E_1, \dots, E_k)$ be a $k$-skeleton and $A = C_0(E_0)$. By a \emph{unitary cocycle for $E$}, we mean a
collection $\{U_{i,j}(v,w) : v,w \in E^0, 1 \le i < j \le k\}$ of unitary operators $U_{i,j}(v,w)
\in \mathcal{U}(\C^{v E^1_j E^1_i w}, \C^{v E^1_i E^1_j w})$ such that the isomorphisms $U_{i,j} :
X(E_j) \otimes_A X(E_i) \to X(E_i) \otimes_A X(E_j)$ for $i < j$ given by Lemma~\ref{lem:U's
induce iso} satisfy the cocycle identity
\begin{equation}\label{eq:cocycle id}
\begin{split}
(U_{i,j} \otimes 1_{X_l})&(1_{X_j} \otimes U_{i,l})(U_{j,l} \otimes 1_{X_i})\\
    &= (1_{X_i} \otimes U_{j,l})(U_{i,l} \otimes 1_{X_j})(1_{X_l} \otimes U_{i,j})\quad\text{ for all $1 \le i < j < l \le k$.}
\end{split}
\end{equation}
\end{defn}

\begin{rmk}\label{rmk:k=2}
If $k = 2$ then the condition~\eqref{eq:cocycle id} is vacuous because the inequality $1 \le i < j
< l \le 2$ has no solutions. So a unitary cocycle for a $2$-skeleton $(E_1, E_2)$ is nothing more
than a collection $\{U_{1,2}(v,w) : v,w \in E^0\}$ of unitary operators $U_{1,2}(v,w) \in
\mathcal{U}(\C^{v E^1_2 E^1_1 w}, \C^{v E^1_1 E^1_2 w})$. In particular, since each $|v E^1_2
E^1_1 w| = |v E^1_1 E^1_2 w|$, every $2$-skeleton admits many unitary cocycles.

By contrast, \cite[Example~5.15(ii)]{KPS2} presents an example, due to Jack Spielberg, of a
$3$-skeleton that cannot be the skeleton of a $3$-graph, and it is straightforward to extend the arguments used in that example to see that  this $3$-skeleton does not admit any unitary cocycles. In
particular, for $k \ge 3$ the existence of a unitary cocycle for $(E_1, \dots, E_k)$ is a
nontrivial additional hypothesis in our results henceforth. However, Proposition~\ref{prp:k-graph
to cocycle} below shows that there are plenty of $k$-skeletons that do admit unitary
cocycles.
\end{rmk}

We are interested in unitary cocycles because they correspond exactly with product systems over
$\N^k$ with generating fibres $X_{e_i} = X(E_i)$. To be precise, the results of  \cite{fs} (in particular \cite[Remark 2.3]{fs}) imply that every unitary cocycle determines a product system.  This is the content of the next lemma.

\begin{lemma}\label{lem:cocycle->prodsys}
Let $E = (E_1, \dots, E_k)$ be a $k$-skeleton, and let $A = C_0(E^0)$.
\begin{enumerate}
\item Let $U = \{U_{i,j}(v,w) : v,w \in E^0, 1 \le i < j \le k\}$ be a unitary cocycle for $E$.
    There is a unique product system $X_U$ over $\N^k$ with coefficient algebra $A$ such that:
    $X_{e_i} = X(E_i)$ for $i \le k$; for all $1 \le i < j \le k$, we have $X_{e_i + e_j} =
    X_{e_i} \otimes_A X_{e_j}$; and the multiplication in $X$ satisfies, for $1 \le i < j \le
    k$, $\xi \in X_{e_i}$ and $\eta \in X_{e_j}$,
    \begin{equation}\label{eq:the multiplication}
        \xi \eta = \xi \otimes \eta\quad\text{ and }\quad \eta \xi = U_{i,j}(\eta \otimes \xi).
    \end{equation}
\item Suppose that $X$ is a product system over $\N^k$ with coefficient algebra $A$ such that
    $X_{e_i} \cong X(E_i)$ for all $i \le k$. Then there exists a unitary cocycle $U$ for $E$
    such that $X \cong X_U$.
\end{enumerate}
\end{lemma}
\begin{proof}
(1) This follows immediately from \cite[Theorem~2.1]{fs}.

(2) Let $W : X_{e_2} \otimes_A X_{e_1} \to X_{e_1} \otimes_A X_{e_2}$ be the isomorphism
determined by the composition of the multiplication map $X_{e_2} \otimes_A X_{e_1} \to X_{(1,1)}$
with the inverse of the multiplication map $X_{e_1} \otimes_A X_{e_2} \to X_{(1,1)}$. Write
$E_{12}$ for the graph $(E^0, E^1_1E^1_2, r, s)$ and $E_{21}$ for the graph $(E^0, E^1_2E^1_1, r,
s)$. By identifying each $X_{e_i}$ with $X(E_i)$, we can regard $W$ as an isomorphism $X(E_2)
\otimes_A X(E_1) \to X(E_1) \otimes_A X(E_2)$. A straightforward calculation shows that there are
isomorphisms $M_{i,j} : X(E_i) \otimes_A X(E_j) \to X(E_{ij})$ such that $M_{i,j}(\delta_e \otimes
\delta_f) = \delta_{ef}$ whenever $s(e) = r(f)$. So we obtain an isomorphism $\widetilde{W} :
X(E_{21}) \to X(E_{12})$ given by $\widetilde{W} = M_{1,2} \circ W \circ M_{2,1}^{-1}$. Since this
is an isomorphism of Hilbert modules, it restricts to isomorphisms $\widetilde{W}(v,w) : \delta_v
\cdot X(E_{21}) \cdot \delta_w \to \delta_v \cdot X(E_{12}) \cdot \delta_w$. These spaces are
canonically isomorphic to $\C^{vE^1_2E^1_1 w}$ and $\C^{vE^1_1E^1_2 w}$, and so
$\widetilde{W}(v,w)$ determines a unitary $U_{1,2}(v,w) \in \mathcal{U}(\C^{vE^1_2E^1_2 w},
\C^{vE^1_1E^1_2 w})$.

The isomorphism $U_{1,2} : X(E_2) \otimes_A X(E_1) \to X(E_1) \otimes_A X(E_2)$ obtained from the maps
$\{ U_{1,2}(v,w)\}_{v,w}$ as in Lemma~\ref{lem:U's induce iso} satisfies $M_{1,2}(U_{1,2}(\delta_f \otimes
\delta_e)) = \widetilde{W}(M_{2,1}(\delta_f \otimes \delta_e))$ whenever $s(f) = r(e)$, and
therefore $U_{1,2} = W$. So the uniqueness result \cite[Theorem~2.2]{fs} shows that the product
system $X$ is isomorphic to $X_U$ as claimed.
\end{proof}

Our next result shows that for any row-finite $k$-graph $\Lambda$ with no sources, and for any
2-cocycle on $\Lambda$, we obtain a unitary cocycle such that the twisted $k$-graph $C^*$-algebra
coincides with the Cuntz--Nica--Pimsner algebra of the product system over $\N^k$ determined by
the unitary cocycle.

\begin{prop}\label{prp:k-graph to cocycle}
Let $\Lambda$ be a row-finite $k$-graph with no sources and suppose that $\phi :
\bigsqcup_{i<j\le k} \Lambda^{e_i+e_j} \to \T$ is a cubical 2-cocycle. For $i \le k$, let $E_i$
be the directed graph $E_i = (\Lambda^0, \Lambda^{e_i}, r, s)$. Then $E = (E_1, \dots, E_k)$ is
a $k$-skeleton. There is a unitary cocycle $U = U^{\Lambda,\phi}$ for $E$ such that for all $v,w
\in \Lambda^0$ and $1 \le i < j \le k$, and for all $e,e' \in E^1_i$ and $f,f' \in E^1_j$
satisfying $ef = f'e'$ in $\Lambda$, we have
\[
U_{i,j}(v,w)(\delta_{f'e'}) = \phi(ef)\delta_{ef}.
\]
Let $X_U$ be the product system associated to $U$ in Lemma~\ref{lem:cocycle->prodsys}. Then
there is an isomorphism $\pi : C^*_\phi(\Lambda) \cong \mathcal{NO}_{X_U}$ such that $\pi(p_v) =
i_0(\delta_v)$ for all $v \in \Lambda^0$, and such that $\pi(s_\lambda) i_{e_i}(\delta_\lambda)$
for all $i \le k$ and $\lambda \in \Lambda^{e_i}$.
\end{prop}
\begin{proof}
The factorisation rules in $\Lambda$ determine bijections $f_{ij} : v\Lambda^{e_i} \Lambda^{e_j} w
\to v\Lambda^{e_i+e_j} w$, and hence bijections $F_{ij} := f_{ij}^{-1} \circ f_{ji} :
v\Lambda^{e_j}\Lambda^{e_i}w \to v\Lambda^{e_i}\Lambda^{e_j}w$ for $i < j$, such that if $e, e' \in \Lambda^{e_j} = E^1_i$ and $f, f' \in \Lambda^{e_i} = E^1_j$ satisfy $ef = f'e'$, then $F_{ij}(ef)  = f'e'.$ In particular, $E$
is a $k$-skeleton. The maps $F_{ij}$ induce a unitary (in fact a permutation matrix) $V_{ij}(v,w) : \C^{v E^1_jE^1_i w} \to
\C^{vE^1_iE^1_j w}$ for each $v, w \in E^0$, and post-composing $V_{i,j}(v,w)$ with the diagonal unitary
$\operatorname{diag}(\phi(\lambda))_{\lambda \in vE^1_iE^1_j w}$ yields a unitary $U(v,w)_{i,j}$
satisfying the desired formula. If $e f g  \in E^1_i E^1_j E^1_l$ with $1 \le l < j < i \le k$,
and if the factorisation rules give $e f g= f^1e^1g = f^1g^1e^2 = g^2f^2e^2$ and also $e f g=
eg_1f_1 = g_2e_1f_1 = g_2f_2e_2$, then associativity of composition in $\Lambda$ gives $e^2 = e_2$, $f^2 = f_2$ and $g^2 = g_2$:
\[
\begin{tikzpicture}[scale=2]
 \node (00) at (0,0) {};
    \node (01) at (0,1) {};
    \node (10) at (1,0) {};
    \node (11) at (1,1) {};
    \node (53) at (0.5,0.3) {};
    \node (513) at (0.5, 1.3) {};
    \node (1513) at (1.5, 1.3) {};
    \node (153) at (1.5, 0.3) {};
    \draw[-stealth] (01)-- node[left] {$f^1$} (00);
    \draw[-stealth, dotted] (10)-- node[below] {$e$} (00);
    \draw[-stealth] (11)-- node[left, near start] {$f$} (10);
    \draw[-stealth, dotted] (11)-- node[above, near start ] {$e^1$} (01);
    \draw[-stealth, dashed] (1513) -- node[left] {$g$} (11);
    \draw[-stealth, dashed] (513) -- node[left] {$g^1$} (01);
    \draw[-stealth, dashed] (153) -- node[right] {$g_1$} (10);
    \draw[-stealth, dashed] (53) -- node[right] {$g^2$} (00);
    \draw[-stealth,dotted] (1513) -- node[above] {$e^2$} (513);
    \draw[-stealth] (513) -- node[left, near end] {$f^2$} (53);
    \draw[-stealth,dotted] (153) -- node[above, near start] {$e_1$} (53);
    \draw[-stealth] (1513) -- node[right] {$f_1$} (153);
\end{tikzpicture}
\]

Quick calculations  then give
\begin{align*}
(U_{i,j} \otimes 1_{X_l})(1_{X_j} \otimes U_{i,l})(U_{j,l} \otimes 1_{X_i})(\delta_e \otimes \delta_f \otimes \delta_g)
    &= \phi (ef)\phi(e^1g)\phi(f^1g^1) \delta_{g^2} \otimes \delta_{f^2} \otimes \delta_{e^2},\\
\intertext{and}
(1_{X_i} \otimes U_{j,l})(U_{i,l} \otimes 1_{X_j})(1_{X_l} \otimes U_{i,j})(\delta_e \otimes \delta_f \otimes \delta_g)
    &= \phi (fg)\phi(eg_1)\phi (e_1f_1) \delta_{g_2} \otimes \delta_{f_2} \otimes \delta_{e_2} \\
    &=  \phi (fg)\phi(eg_1)\phi (e_1f_1) \delta_{g^2} \otimes \delta_{f^2} \otimes \delta_{e^2}.
\end{align*}
The cubical cocycle condition shows that $c(ef)c(e^1g)c(f^1g^1) = c(fg)c(eg_1)c(e_1f_1)$. Thus
condition~\eqref{eq:cocycle id} holds when both sides are applied to an elementary tensor
$\delta_e \otimes \delta_f \otimes \delta_g$ corresponding to a path $e f g \in E^1_i E^1_j
E^1_l$. Since these are the only elementary tensors of basis vectors that are nonzero, they span
$X_i \otimes_A X_j \otimes_A X_k$, and therefore $U$ is a unitary cocycle for $E$ as claimed.

For $v \in \Lambda^0$, define $P_v := i_0(\delta_v)$ and for $\lambda \in \Lambda^{e_i}$ define
$S_\lambda := i_{e_i}(\delta_\lambda)$. We claim that these form a Cuntz--Krieger
$\phi$-representation of $\Lambda$ as in \cite[Definition~7.4]{kps3}. The $P_v$ are mutually
orthogonal projections because the $\delta_v$ are. For $\lambda \in \Lambda^{e_i}$, we have
\[
S_\lambda^* S_\lambda
    = i_{e_i}(\delta_\lambda)^*i_{e_i}(\delta_\lambda)
    = i_0(\langle \delta_\lambda, \delta_\lambda\rangle_A)
    = i_0(\delta_s(\lambda))
    = P_{s(\lambda)}
\]
which is condition~(1) of \cite[Definition~7.4]{kps3}.

To check condition~(2), fix $1 \le i < j \le k$ and fix $\mu,\mu' \in \Lambda^{e_i}$ and $\nu,\nu'
\in \Lambda^{e_j}$ such that $\mu\nu = \nu'\mu'$. Using~\eqref{eq:the multiplication} at the
second and fifth equalities, and using the definition of $U_{r(\nu'), s(\mu')}$ at the fourth, we
calculate:
\begin{align*}
s_\nu' s_\mu'
    &= i_{e_2}(\delta_{\nu'})i_{e_1}(\delta_{\mu'})
    = i_{(1,1)}(U(\delta_{\nu'}\otimes \delta_{\mu'}))
    = i_{(1,1)}(U_{r(\nu'), s(\mu')}(\delta_{\nu'}\otimes \delta_{\mu'}))\\
    &= i_{(1,1)}(\phi(\mu\nu)\delta_{\mu}\otimes \delta_{\nu})
    = \phi(\mu\nu)i_{e_1}(\delta_{\mu})i_{e_2}(\delta_{\nu})
    = \phi(\mu\nu) s_\mu s_\nu.
\end{align*}

Finally, to check condition~(3), fix $v \in E^0$. Writing $\alpha : A \to \K(X_i)$ for the
homomorphism implementing the left action, we have $\sum_{e \in v\Lambda^{e_i}} \theta_{\delta_e,
\delta_e} = \alpha(\delta_v)$, and so the Cuntz--Pimsner covariance condition gives $\sum_{e \in
v\Lambda^{e_i}} S_e S^*_e = i^{(e_i)}\big(\sum_e \theta_{\delta_e, \delta_e}\big) = i_0(\delta_v)
= P_v$. Thus the operators $\{P_v\}_{v \in \Lambda^0} $ and $\{S_\lambda\}_{\lambda \in \Lambda^{e_i}}$ form a Cuntz--Krieger $\phi$-representation of $\Lambda$ as
claimed.

The universal property of $C^*_\phi(\Lambda)$ now gives a homomorphism $\pi : C^*_\phi(\Lambda)
\to \mathcal{NO}_{X}$ that carries $p_v$ to $P_v$ and $s_\lambda$ to $S_\lambda$ for $\lambda \in
\bigsqcup_{i \le k} \Lambda^{e_i}$. Since $A$ is spanned by the $\delta_v$ and each $X_i$ is
spanned by $\{\delta_{\lambda} : \lambda \in \Lambda^{e_i}\}$ we see that the range of $\pi$
contains $C^*(i_0(A) \cup \bigcup_i i_{e_i}(X_i))$. Since multiplication gives isomorphisms $X_n
\cong X_1^{\otimes n_1} \otimes_A X_2^{\otimes n_2} \otimes_A \cdots \otimes_A X_k^{\otimes n_k}$,
we see that
\[
i_n(X_n) = \clsp\{S_{\lambda^1_1} \dots S_{\lambda^1_{n_1}} S_{\lambda^2_1} \dots S_{\lambda^2_{n_2}} \dots S_{\lambda^k_1} \dots S_{\lambda^k_{n_k}} :
    \lambda^i_j \in \Lambda^{e_i}\}
\]
and so is contained in the image of $\pi$. Since $i_0$ is injective \cite[Lemma~3.5 and
Theorem~4.1]{sims--yeend}, the $P_v$ are all nonzero. Using \cite[Theorem~4.15 and
Proposition~5.7]{KPS4} we see that there is an isomorphism $\omega : C^*_\phi(\Lambda) \cong
C^*(\Lambda, c)$, for some categorical 2-cocycle $c$, such that $\omega$ carries $p_v$ to $p_v$ for each $v \in
\Lambda^0$, and then an application of the gauge-invariant uniqueness theorem
\cite[Corollary~7.7]{KPS4} shows that $\pi \circ \omega^{-1}$ is injective, and therefore $\pi$
is injective. So $C^*_\phi(\Lambda) \cong \mathcal{NO}_{Y^0}$ as claimed.
\end{proof}

Lemma~\ref{lem:cocycle->prodsys} also allows us to construct homotopies of product systems using
continuous paths of unitary cocycles.

\begin{defn}
\label{def:path of unitaries}
Let $E = (E_1, \dots, E_k)$ be a $k$-skeleton. A \emph{continuous path of unitary cocycles for
$E$} is a family
\[
    \{U^t_{i,j}(v,w) : v,w \in E^0, 1 \le i < j \le k, t \in [0,1]\}
\]
such that for each $t$, the family $\{U^t_{i,j}(v,w) : v,w \in E^0, 1 \le i < j \le k\}$ is a
unitary cocycle for $E$, and such that for each $i < j$ and each $v,w \in E^0$, the function $t
\mapsto U^t_{i,j}(v,w)$ is a continuous function from $[0,1]$ to $\mathcal{U}(\C^{vE^1_jE^1_i
w}, \C^{vE^1_iE^1_j w})$.
\end{defn}

\begin{prop}\label{prp:path of cocycles}
Let $E = (E_1, \dots, E_k)$ be a $k$-skeleton, and suppose that the family $U = \{U^t_{i,j}(v,w) :
v,w \in E^0, 1 \le i < j \le k, t \in [0,1]\}$ is a continuous path of unitary cocycles for $E$.
Let $A = C_0(E^0)$, and for each $i$, let $X_i = X(E_i)$ the graph bimodule for $E_i$. Define $B
:= C([0,1], A)$ and for each $i$, define $Y_i := C([0,1], X_i)$ regarded as a Hilbert $B$-bimodule
under the actions $(f \cdot F \cdot g)(t) = f(t)\cdot F(t)\cdot g(t)$ and $\langle F,
G\rangle_B(t) = \langle F(t), G(t)\rangle_A$. Then there is a unique product system $Y$ over
$\N^k$ with coefficient algebra $B$ such that: $Y_{e_i} = Y_i$ for $i \le k$; for all $1 \le i < j
\le k$, we have $Y_{e_i + e_j} = Y_i \otimes_B Y_j$; and multiplication in $Y$ satisfies
\begin{equation}\label{eq:the multiplication with t}
(F G)(t) = F(t)  \otimes_A G(t) \quad\text{ and }\quad (G F)(t) = U^t_{i,j}(G(t) \otimes F(t))
\end{equation}
whenever $1 \le i < j \le k$, $F \in Y_i$, $G \in Y_j$ and $t \in [0,1]$. The product system $Y$
is a homotopy of product systems from $X_{U^0}$ to $X_{U^1}$.
\end{prop}
\begin{proof}
Since each $U^t_{i,j}$ gives an isomorphism $X_j \otimes_A X_i \to X_i \otimes_A X_j$, and $U$ is
a continuous path of unitary cocycles, the map $G \otimes F \mapsto (t \mapsto U_{i,j}^t(G(t)
\otimes F(t)))$ determines an isomorphism $U_{i,j} : Y_j \otimes_B Y_i \to Y_i \otimes_B Y_j$.
Since the $U_{i,j}^t$ satisfy the cocycle identity, so do the induced isomorphisms $U_{i,j}$. So
the first statement follows from \cite[Theorem~2.1]{fs}.

The formulas for the actions show immediately that $Y$ is nondegenerate, and is fibred over
$[0,1]$. Thus, it suffices to show that $Y^t = X_{U^t}$.  Let $I_t$ be the ideal of $B$ generated
by $C_0([0,1] \setminus \{t\}) \subseteq C([0,1])$. For $F, G \in Y_i$, we have $F  + Y_i\cdot I_t
= G + Y_i\cdot I_t$ if and only if $F(t) = G(t)$. Routine calculations then show that the map
$\eta_t : F  + Y_i\cdot I_t \mapsto F(t)$ is a bimodule morphism (in fact an isomorphism) from
$Y_i^t$ to $X_i$. Since each $Y_i$ is fibred over $[0,1]$ we see that
 $(Y_i \otimes_B Y_j) \cdot I_t = (Y_i \cdot I_t) \otimes_B (Y_j \cdot I_t)$ for all $i, j$ and $t$, and that
$\eta_t$ induces an isomorphism (also denoted $\eta_t$) from $Y_i^t \otimes_A Y_j^t$
to $X_i \otimes_A X_j$. These isomorphisms satisfy
\begin{align*}
U^t_{i,j}(\eta_t((G + Y_j \cdot I_t) \otimes
(F + Y_i \cdot I_t)))
    &= U^t_{i,j}(G(t) \otimes
    F(t))
    = (U_{i,j}(G \otimes F))(t)\\
    &= \eta_t(U_{i,j}(G \otimes F) + (Y_i \otimes_B Y_j) \cdot I_t),
\end{align*}
and so it follows from the uniqueness theorem \cite[Theorem~2.2]{fs} that the product system
$Y^t$ is isomorphic to $X_{U^t}$. In particular $Y^0 \cong X_{U^0}$ and $Y^1 \cong X_{U^1}$.
\end{proof}

\begin{cor}\label{cor:connected->K-theory}
Let $E = (E_1, \dots, E_k)$ be a $k$-skeleton. If $U = \{U^t_{i,j}(v,w)\}$ is a continuous path of
unitary cocycles for $E$, then $K_*(\mathcal{NO}_{X_{U^0}}) \cong K_*(\mathcal{NO}_{X_{U^1}})$. In
particular, suppose that  $\Lambda_1$ and $\Lambda_2$ are $k$-graphs with the same skeleton and
that $\phi_l : \bigsqcup_{i < j \le k} \Lambda_l^{e_i + e_j} \to \T$ is a cubical cocycle for $i =
1,2$. If there is a continuous path of unitary cocycles from $U_{\phi_1}$ to $U_{\phi_2}$, then
$K_*(C^*_{\phi_1}(\Lambda_1)) \cong K_*(C^*_{\phi_2}(\Lambda_2))$.
\end{cor}
\begin{proof}
Proposition~\ref{prp:path of cocycles} gives a homotopy of product systems from $X_{U^0}$ to
$X_{U^1}$, and then Theorem~\ref{thm:htopy-KK} implies that $K_*(\mathcal{NO}_{X_{U^0}}) \cong
K_*(\mathcal{NO}_{X_{U^1}})$ as required. The final statement follows from the first statement and
Proposition~\ref{prp:k-graph to cocycle}.
\end{proof}

The next result demonstrates that Corollary~\ref{cor:connected->K-theory} is particularly useful
when $k = 2$.

\begin{prop}\label{prp:unitary cocycles connected k=2}
Let $E = (E_1, E_2)$ be a $2$-skeleton. Then the space of unitary cocycles for $E$ is
path-connected in the sense that for any two unitary cocycles $V, W$ for $E$ there is a continuous
path $U^t$ of unitary cocycles for $E$ such that $U^0 = V$ and $U^1 = W$.
\end{prop}
\begin{proof}
Fix $v,w \in E^0$. The space
\[
    \mathcal{U}(\C^{vE^1_jE^1_iw}, \C^{vE^1_iE^1_jw})
\]
is homeomorphic to the finite-dimensional unitary group $\mathcal{U}_{|v\Lambda^{(1,1)}w|}$, and
therefore path-connected. So we can choose a continuous path $U^t_{1,2}(v,w)$ in
$\mathcal{U}(\C^{vE^1_jE^1_iw}, \C^{vE^1_iE^1_jw})$ such that $U^0_{1,2}(v,w) = V_{1,2}(v,w)$
and such that $U^1_{1,2}(v,w) = W_{1,2}(v,w)$. Since $k = 2$, the resulting isomorphisms
$U^t_{i,j} : X(E_j) \otimes_A X(E_i) \to X(E_i) \otimes_A X(E_j)$ indexed by $1 \le i < j \le 2$
all vacuously satisfy the condition~\eqref{eq:cocycle id} (see Remark~\ref{rmk:k=2}). So the
$U^t_{i,j}(v,w)$ constitute a continuous path of unitary cocycles for $E$.
\end{proof}

We can now prove the main theorem of the section.

\begin{proof}[Proof of Theorem~\ref{thm:twisted 2-graph K-th}]
For $i = 1,2$ let $E_i$ be the graph $(\Lambda^0, \Lambda^{e_i}, r, s)$.
Proposition~\ref{prp:k-graph to cocycle} shows that there are product systems $X^{\Lambda,\phi}$
and $X^{\Gamma,\psi}$ over $\N^2$ with coefficient algebra $C_0(E^0)$ such that
$X^{\Lambda,\phi}_{e_i} \cong X^{\Gamma,\psi}_{e_i} \cong X(E_i)$ for $i = 1,2$ and such that
$\mathcal{NO}_{X^{\Lambda,\phi}} \cong C^*_\phi(\Lambda)$ and $\mathcal{NO}_{X^{\Gamma,\psi}}
\cong C^*_\psi(\Gamma)$. By Lemma~\ref{lem:cocycle->prodsys}, there are unitary cocycles
$U^{\Lambda, \phi}, U^{\Gamma, \psi}$ for $E = (E_1, E_2)$ such that $X^{\Lambda, \phi} \cong
X_{U^{\Lambda, \phi}}$ and $X^{\Gamma, \psi} \cong X_{U^{\Gamma, \psi}} $.
Proposition~\ref{prp:unitary cocycles connected k=2} then shows that there is a continuous path of
unitaries connecting $U^{\Lambda, \phi}$ and $ U^{\Gamma, \psi}$. Hence
Corollary~\ref{cor:connected->K-theory} gives the result.
\end{proof}

\begin{rmk}
Let $E = (E_1, E_2)$ be a $2$-skeleton. By \cite[Section~6]{kp2} (see also \cite[Remark~2.3]{fs}),
any range-and-source-preserving bijection $E^1_2E^1_1 \to E^1_1E^1_2$ determines a $2$-graph with
skeleton $E$. In particular, there is at least one such $2$-graph $\Lambda$. Let $M_1$ and $M_2$
be the $E^0 \times E^0$ adjacency matrices of $E_1$ and $E_2$; that is $M_i(v,w) = |v E^1_i w|$.
We regard the transpose matrices $M^T_i$ as homomorphisms $M^T_i : \Z E^0 \to \Z E^0$, and then we
obtain homomorphisms
\[
\big(1 - M^t_1, 1 - M^t_2\big) : \Z E^0 \oplus \Z E^0 \to \Z E^0\quad\text{ and }\quad
\left(\begin{matrix} M^t_2 - 1 \\ 1 - M^t_1\end{matrix}\right) : \Z E^0 \to \Z E^0 \oplus \Z E^0
\]
Evans' calculation \cite[Proposition~3.16]{Evans} of the $K$-theory of the $C^*$-algebra of a
row-finite $2$-graph shows that
\begin{equation}\label{eq:K-th}
\begin{split}
K_0(C^*(\Lambda)) &\cong \operatorname{coker}\big(1 - M^t_1, 1 - M^t_2\big) \oplus \ker\left(\begin{matrix} M^t_2 - 1 \\ 1 - M^t_1\end{matrix}\right),\quad\text{ and}\\
K_1(C^*(\Lambda)) &\cong \ker\big(1 - M^t_1, 1 - M^t_2\big)/\operatorname{image}\left(\begin{matrix} M^t_2 - 1 \\ 1 - M^t_1\end{matrix}\right).
\end{split}
\end{equation}
Combining this with Proposition~\ref{prp:k-graph to cocycle} and Theorem~\ref{thm:twisted 2-graph
K-th}, we deduce that if $\Gamma$ is any 2-graph with skeleton $E$ and $\phi$ is any 2-cocycle on
$\Gamma$, then $K_*(C^*_\phi(\Gamma) \cong K_*(C^*(\Lambda))$ is given by the
formulas~\eqref{eq:K-th}.
\end{rmk}

We can also say a little about the situation of single-vertex $k$-graphs. Although we do not
have an entirely satisfactory result in this situation, we indicate what we can say since there
has been some interest (see Corollary~5.8, Remark~5.9, and the following paragraph of
\cite{BOS}) in deciding whether the $K$-theory of the $C^*$-algebra of a $1$-vertex $k$-graph is
independent of the factorisation rules.

Fix $k \ge 0$ and integers $n_1, \dots, n_k \ge 1$. Suppose that, for each $1 \le i < j \le k$,
we have a unitary transformation $U_{i,j} : \C^{n_j} \otimes \C^{n_i} \to \C^{n_i} \otimes
\C^{n_j}$. We say that the system $U_{i,j}$ is a \emph{unitary cocycle for $(n_1,\dots, n_k)$}
if, for all $1 \le i < j < l \le k$, we have
\begin{equation}
\label{eq:cocycle-mcs}
(U_{i,j} \otimes 1_{n_l})(1_{n_j} \otimes U_{i,l})(U_{j,l} \otimes 1_{n_i})
    = (1_{n_i} \otimes U_{j,l})(U_{i,l} \otimes 1_{n_j})(1_{n_l} \otimes U_{i,j}).
\end{equation}
A {\em continuous path of unitary cocycles for $(n_1, \ldots, n_k)$} is a family $\{U^t_{i,j}:
t\in [0,1], 1 \leq i < j \leq k \}$ of unitary cocycles for $(n_1, \ldots, n_k)$ such that the map
$t \mapsto U^t_{i,j}$ is continuous for each $1 \leq i < j \leq k$. These definitions are the
translation of Definitions \ref{def:unitary cocycle}~and~\ref{def:path of unitaries} to the
simpler setting of $k$-skeletons with one vertex and $n_i$ edges of colour $i$.

\begin{example}\label{ex:k-graph unitary cocycle}
Let $\Lambda$ be a $k$-graph with a single vertex $v$ and satisfying $|\Lambda^{e_i}| = n_i$ for
$1 \le i \le k$. Let $\phi : \bigsqcup_{i \not= j} \Lambda^{e_i + e_j} \to \T$ be a cubical
$2$-cocycle. Identifying $\C^{n_i} \otimes \C^{n_j}$ with $\C^{\Lambda^{e_i}\Lambda^{e_j}}$ for
all $i \not= j$, we obtain from Proposition~\ref{prp:k-graph to cocycle} a unitary cocycle
$U^{\Lambda,\phi}$ for $(n_1, \dots, n_k)$.
\end{example}

\begin{cor} Fix integers $ k \ge 0$ and $n_1, \ldots, n_k \geq 1$.
\begin{enumerate}
\item Let $\Lambda_1$ and $\Lambda_2$ be single-vertex $k$-graphs such that $|\Lambda^{e_i}_1| =
    |\Lambda^{e_i}_2| =: n_i$ for all $i$. If there is a continuous path of unitary cocycles for
    $(n_1, \dots, n_k)$ from $U^{\Lambda_1, 1}$ to $U^{\Lambda_2, 1}$, then $K_*(C^*(\Lambda_1))
    \cong K_*(C^*(\Lambda_2))$.
\item If the space of all unitary cocycles for $(n_1, \dots, n_k)$ is path-connected, then we
    have $K_*(C^*(\Lambda_1, c_1)) \cong K_*(C^*(\Lambda_2, c_2))$ for any single vertex
    $k$-graphs $\Lambda_i$ with $|\Lambda_i^{e_j}| = n_j$ and any categorical 2-cocycles $c_i$
    on $\Lambda_i$.
\end{enumerate}
\end{cor}
\begin{proof}
This follows immediately from Corollary~\ref{cor:connected->K-theory}.
\end{proof}

Unfortunately, we do not know whether the space of unitary cocycles for $(n_1, \ldots, n_k)$ is connected: certainly
given any two unitary cocycles $U$ and $V$, we can find continuous paths $W^t_{i,j}$ from
$U_{i,j}$ to $V_{i,j}$ for all $i,j$, but it is not at all clear that these can be chosen so
that, for each $t$, $\{ W^t_{i,j}: 1\leq i < j \leq k\}$ satisfies the cocycle condition \eqref{eq:cocycle-mcs}. This seems closely related to the question of whether
the space of solutions to the Yang--Baxter equation is path-connected (see \cite{Yang}).

\begin{rmk}
More generally, if $E = (E_1, \dots, E_k)$ is a $k$-skeleton such that the space of unitary
cocycles for $E$ is path-connected, then for any two $k$-graphs $\Lambda$ and $\Gamma$ with common
skeleton $\Lambda^{e_i} = E_i = \Gamma^{e_i}$, and any cubical cocycles $\phi$ for $\Lambda$ and
$\psi$ for $\Gamma$, we have $K_*(C^*_{\phi}(\Lambda)) \cong K_*(C^*_{\psi}(\Gamma))$. Again,
since each $\C^{v E_j E_i w}$ is finite dimensional, for each $u,v,i,j$ we can find a continuous
path of unitaries from $U^{\Lambda, \phi}_{i,j}(u,v)$ to $U^{\Gamma, \phi}_{i,j}(u,v)$, but it is
not at all clear that these can be chosen to satisfy the cocycle identity~\eqref{eq:cocycle id}.
\end{rmk}


\begin{thebibliography}{88}
\bibitem{BOS} S. Barlak, T. Omland and N. Stammeier, \emph{On the
    {$K$}-theory of {$C^{\ast}$}-algebras arising from integral dynamics}, Ergodic Theory Dynam.
    Systems \textbf{38} (2018), 832--862.

\bibitem{CaHS} L.O. Clark, A. an Huef and A. Sims, \emph{A{F}-embeddability of
    2-graph algebras and quasidiagonality of {$k$}-graph algebras}, J. Funct. Anal. \textbf{271}
    (2016), 958--991.

\bibitem{elliott} G.A. Elliott, \emph{On the {$K$}-theory of the {$C^{\ast} $}-algebra
    generated by a projective representation of a torsion-free discrete abelian group}, Monogr.
    Stud. Math., 17, Operator algebras and group representations, {V}ol. {I} ({N}eptun, 1980),
    157--184, Pitman, Boston, MA, 1984.

\bibitem{Evans} D.G. Evans, \emph{On the $K$-theory of higher-rank graph {$C^*$}-algebras}, New
    York J. Math. \textbf{14} (2008), 1--31.

\bibitem{EvansSims} D.G. Evans and A. Sims, {\em When is the {C}untz-{K}rieger algebra of a higher-rank graph approximately finite-dimensional?}, J. Funct. Anal. {\bf 263} (2012), 183--215.

\bibitem{fletcher-thesis-1} J. Fletcher, Iterating the Cuntz--Nica--Pimsner construction for
    product systems, PhD thesis, University of Wollongong, 2017.

\bibitem{fletcher-NYJM} J. Fletcher, \emph{Iterating the {C}untz-{N}ica-{P}imsner construction for
              compactly aligned product systems}, {New York J. Math.} {\bf 24} (2018), {739--814}.

\bibitem{Fowler} N.J. Fowler, \emph{Discrete product systems of {H}ilbert bimodules}, Pacific J.
    Math. \textbf{204} (2002), 335--375.

\bibitem{FR} N.J. Fowler and I. Raeburn, \emph{The {T}oeplitz algebra of a
    {H}ilbert bimodule}, Indiana Univ. Math. J. \textbf{48} (1999), 155--181.

\bibitem{fs} N.J. Fowler and A. Sims, \emph{Product systems over right-angled {A}rtin
    semigroups}, Trans. Amer. Math. Soc. \textbf{354} (2002), 1487--1509.

\bibitem{gillaspy-PJM} E. Gillaspy, {\em {$K$}-theory and homotopies of 2-cocycles on higher-rank graphs}, Pacific J. Math. \textbf{278} (2015),
407--426.

\bibitem{GW} E. Gillaspy and J. Wu, \emph{Isomorphism of the cubical and categorical
    cohomology groups of a higher-rank graph }, preprint 2018 (arXiv:1807.02245 [math.OA]).

\bibitem{HRSW} R. Hazlewood, I. Raeburn, A. Sims, and S.B.G. Webster, {\em Remarks on some fundamental results about higher-rank graphs and their $C^*$-algebras}, Proc. Edinburgh Math. Soc. {\bf 56} (2013), 575--597.


\bibitem{jimbo} M. Jimbo, \emph{Introduction to the Yang-Baxter equation}, Int. J. Modern
    Phys. A. \textbf{4} (1989), 3759--3777.

\bibitem{Kaad} J. Kaad, \emph{On the unbounded picture of $KK$-theory}, preprint 2019
    (arXiv:1901.05161v2 [math.KT]).

\bibitem{Katsura} T. Katsura, \emph{On {$C\sp *$}-algebras associated with {$C\sp
    *$}-correspondences}, J. Funct. Anal. \textbf{217} (2004), 366--401.

\bibitem{kp2} A. Kumjian and D. Pask, \emph{Higher rank graph {$C\sp \ast$}-algebras}, New York
    J. Math. \textbf{6} (2000), 1--20.

\bibitem{KPR} A. Kumjian, D. Pask and I. Raeburn, {\em Cuntz--Krieger algebras of directed graphs}, Pacific J. Math. {\bf 184} (1998), 161--174.

\bibitem{KPS2} A. Kumjian, D. Pask and A. Sims, \emph{Generalised morphisms of
    $k$-graphs: $k$-morphs}, Trans. Amer. Math. Soc. \textbf{363} (2011), 2599--2626.

\bibitem{kps3} A. Kumjian, D. Pask and A. Sims, \emph{Homology for higher-rank graphs and twisted
    {$C^\ast$}-algebras}, J. Funct. Anal. \textbf{263} (2012), 1539--1574.

\bibitem{KPS4} A. Kumjian, D. Pask and A. Sims, \emph{On twisted higher-rank graph
    {$C^*$}-algebras}, Trans. Amer. Math. Soc. \textbf{367} (2015), 5177--5216.

\bibitem{KPS5} A. Kumjian, D. Pask and A. Sims, \emph{On the {$K$}-theory of
    twisted higher-rank-graph {$C^*$}-algebras}, J. Math. Anal. Appl. \textbf{401} (2013),
    104--113.

\bibitem{KPS7} A. Kumjian, D. Pask and A. Sims, \emph{Graded {$C^\ast$}-algebras, graded
    {$K$}-theory, and twisted {$P$}-graph {$C^\ast$}-algebras}, J. Operator Theory \textbf{80}
    (2018), 295--348.

\bibitem{Lance} E.C. Lance, Hilbert {$C\sp *$}-modules, A toolkit for
    operator algebraists, Cambridge University Press, Cambridge, 1995, x+130.

\bibitem{PRRS} D. Pask, I. Raeburn, M. R{\o}rdam and A. Sims, \emph{Rank-two
    graphs whose {$C\sp *$}-algebras are direct limits of circle algebras}, J. Funct. Anal.
    \textbf{239} (2006), 137--178.

\bibitem{PSS2} D. Pask, A. Sierakowski and A. Sims, \emph{Real rank and
    topological dimension of higher-rank graph algebras}, Indiana Univ. Math. J. \textbf{66}
    (2017), 2137--2168.

\bibitem{Pimsner} M.V. Pimsner, \emph{A class of {$C\sp *$}-algebras generalizing both
    {C}untz--{K}rieger algebras and crossed products by {${\bf Z}$}}, Fields Inst. Commun., 12,
    Free probability theory (Waterloo, ON, 1995), 189--212, Amer. Math. Soc., Providence, RI,
    1997.

\bibitem{Raeburn} I. Raeburn, Graph algebras, Published for the Conference Board of the
    Mathematical Sciences, Washington, DC, 2005, vi+113.

\bibitem{RaeS} I. Raeburn and A. Sims, \emph{Product systems of graphs and the $C^*$-algebras of
    higher-rank graphs}, J. Operator Th. \textbf{53} (2005), 399--429.

\bibitem{tfb} I. Raeburn and D.P. Williams, Morita equivalence
    and continuous-trace {$C\sp *$}-algebras, American Mathematical Society, Providence, RI, 1998,
    xiv+327.

\bibitem{RS} D.I. Robertson and A. Sims, \emph{Simplicity of {$C\sp
    \ast$}-algebras associated to higher-rank graphs}, Bull. Lond. Math. Soc. \textbf{39} (2007),
    337--344.

\bibitem{RobSt} G. Robertson and T. Steger, \emph{Affine buildings, tiling systems and higher
    rank {C}untz--{K}rieger algebras}, J. reine angew. Math. \textbf{513} (1999), 115--144.

\bibitem{ruiz-sims-sorensen}     E. Ruiz, A. Sims and A.P.W. S\o{}rensen, {\em U{CT}-{K}irchberg algebras have nuclear dimension one}, {Adv. Math.} {\bf 279} (2015), 1--28.

\bibitem{Si1} A. Sims, \emph{Gauge-invariant ideals in the {$C\sp *$}-algebras of finitely
    aligned higher-rank graphs}, Canad. J. Math. \textbf{58} (2006), 1268--1290.

\bibitem{Si2} A. Sims, \emph{Relative {C}untz--{K}rieger algebras of finitely aligned
    higher-rank graphs}, Indiana Univ. Math. J. \textbf{55} (2006), 849--868.

\bibitem{sims--yeend} A. Sims and T. Yeend, \emph{{$C^*$}-algebras associated to product systems
    of {H}ilbert bimodules}, J. Operator Theory \textbf{64} (2010), 349--376.

\bibitem{Williams} D.P. Williams, Crossed products of {$C{\sp \ast}$}-algebras, American
    Mathematical Society, Providence, RI, 2007, xvi+528.

\bibitem{Yang} D. Yang, \emph{The interplay between {$k$}-graphs and the {Y}ang--{B}axter
    equation}, J. Algebra \textbf{451} (2016), 494--525.

\end{thebibliography}
\end{document}